\documentclass[english,reqno]{amsart}

\usepackage{amsmath}
\usepackage{amssymb} 
\usepackage[new]{old-arrows}
\usepackage{enumitem} 
\usepackage{mathtools} 
\usepackage[table]{xcolor} 
\usepackage[all]{xy} 
\usepackage{tikz} 
\usepackage{tikz-cd}
\usepackage{indentfirst} 
\usepackage{babel} 
\usepackage{setspace} 
\usepackage{stmaryrd}

\usepackage[colorlinks,linkcolor=red,anchorcolor=green,citecolor=blue]{hyperref} 
\hypersetup{linktocpage = true} 

\usepackage{rotating} 

\usepackage{ytableau} 
\usepackage{longtable} 
\newcolumntype{M}[1]{>{\centering\arraybackslash}m{#1}} 

\usepackage{mathpazo}
\usepackage{mathrsfs}
\DeclareFontFamily{OMS}{rsfs}{\skewchar\font'60}
\DeclareFontShape{OMS}{rsfs}{m}{n}{<-5>rsfs5 <5-7>rsfs7 <7->rsfs10 }{}
\DeclareSymbolFont{rsfs}{OMS}{rsfs}{m}{n}
\DeclareSymbolFontAlphabet{\scr}{rsfs}
\DeclareSymbolFontAlphabet{\scr}{rsfs}

\usepackage[T1]{fontenc}

\def\tr{{\rm tr}}

\newcommand\cE{{\mathcal E}}
\newcommand\cF{{\mathcal F}}
\newcommand\cG{{\mathcal G}}
\newcommand\cH{{\mathcal H}}

\newcommand\cO{{\mathcal O}}

\newcommand\cS{{\mathcal S}}

\newcommand{\ddbar}{\sqrt{-1}\partial\bar\partial}

\newcommand{\rk}{{\rm rk}}

\newcommand{\ID}{{\rm{I}d}}

\theoremstyle{plain}
\newtheorem{thm}{Theorem}[section]
\newtheorem{lemma}[thm]{Lemma}
\newtheorem{prop}[thm]{Proposition}
\newtheorem{cor}[thm]{Corollary}
\newtheorem{defn}[thm]{Definition}
\newtheorem{claim}[thm]{Claim}

\theoremstyle{remark}

\newtheorem{remark}[thm]{Remark}

\newcommand\dbar{{\overline{\partial}}}

\def\dim{\operatorname{dim}}

\def\Tr{\operatorname{Tr}}

\def\max{\operatorname{max}}

\def\Hom{\operatorname{Hom}}

\def\hor{\operatorname{\hor}}
\def\ver{\operatorname{\ver}}

\def\sm{\operatorname{\textsubscript{\rm sm}}}
\def\sing{\operatorname{\textsubscript{\rm sing}}}
\def\orb{\mathrm{orb}}
\def\ver{\operatorname{\textsubscript{\rm ver}}}
\def\hor{\operatorname{\textsubscript{\rm hor}}}

\def\orb{\mathrm{orb}}

\setlist[itemize]{leftmargin=*}
\setlist[enumerate]{leftmargin=*}

\numberwithin{equation}{section} 

\makeatletter

\ifnum\@ptsize=0  \fi \ifnum\@ptsize=2  \fi

\title{Title} 

\subjclass[2010]{}

\author{Fu Xin}
\address{School of sciences, Institute for Theoretic Sciences, Westlake University, Hangzhou, Zhejiang Province, 310030, China}
\email{fuxin54@westlake.edu.cn}

\author{Ou Wenhao}
\address{ Institute of Mathematics, Academy of Mathematics and Systems Science, Chinese Academy of Sciences, Beijing, 100190, China}
\email{wenhaoou@amss.ac.cn}

\begin{document}

\begin{abstract}
In \cite{Ou2024}, the orbifold Bogomolov-Gieseker  inequality is proved for a stable reflexive sheaf on a compact K\"ahler variety with klt singularities.  
In this paper,  we  give a characterization on the stable reflexive sheaf  when the Bogomolov-Gieseker equality holds. 
\end{abstract}

\title{Orbifold Bogomolov-Gieseker inequalities on compact K\"ahler varieties
}

\maketitle

\tableofcontents


\section{Introduction}

The theory of  holomorphic vector bundles is a central object in complex algebraic geometry and complex analytic geometry. 
The notion of stable vector bundles on complete curves was introduced by Mumford  in  \cite{Mumford1963}. 
Such notion of stability was then extended to torsion-free sheaves on any projective manifolds (see \cite{Takemoto1972}, \cite{Gieseker1977}), and is now known as the slope stability. 
An  important property of stable vector bundles is the following Bogomolov-Gieseker inequality, involving the Chern classes of the vector  bundle.
\begin{thm}
\label{thm:BG-inequality-intro}
Let $Z$ be a projective manifold of dimension $n$,   let  $H$ be  an ample divisor, and let $\cF$ be a $H$-stable vector bundle of rank $r$ on $Z$. 
Then 
\[   \Big(c_2(\cF)-\frac{r-1}{2r}c_1(\cF)^2 \Big)  \cdot  H^{n-2} \geq  0. \]
\end{thm}
When  $Z$ is a surface, the inequality was proved in \cite{Bogomolov1978}. 
In higher dimensions, one may apply Mehta-Ramanathan   theorem in  \cite{MehtaRamanathan1981/82}   to  reduce to the case of surfaces, by taking hyperplane  sections. 
Later in \cite{Kawamata1992}, as a part of the proof for the three-dimensional abundance theorem, Kawamata extended the inequality to orbifold Chern  classes of  reflexive sheaves on projective surfaces with quotient singularities. 
The technique of taking hypersurface sections then  allows us to deduce Bogomolov-Gieseker inequalities for reflexive sheaves on projective varieties which have quotient singularities in codimension 2.

On the analytic side, let $(Z,\omega)$ be a compact  K\"ahler manifold, and $(\cF, h)$ a Hermitian  holomorphic vector bundle on $Z$.  
L\"ubke proved that  if $h$ satisfies the Einstein condition, then the following inequality holds (see \cite{Lub82}),
\[   \int_Z \Big(c_2(\cF,h)-\frac{r-1}{2r}c_1(\cF,h)^2 \Big)  \wedge  \omega^{n-2} \geq 0. \] 
It is now well understood that if $\cF$ is slope stable, then it admits a Hermitian-Einstein metric. 
The case when $Z$ is a complete curve was proved by Narasimhan-Seshadri  in \cite{NarasimhanSeshadri1965}, the case of projective surfaces was proved by Donaldson in \cite{Donaldson1985}, and the case of arbitrary compact K\"ahler manifolds was proved by Uhlenbeck-Yau in \cite{UhlenbeckYau1986}.   
Simpson extended the existence of  Hermitian-Einstein  metric to stable Higgs bundles, on compact and certain non compact K\"ahler manifolds, see \cite{Simpson1988}. 
Furthermore,  in \cite{BandoSiu1994},  Bando-Siu introduced the notion of admissible metrics and proved the existence of admissible  Hermitian-Einstein  metrics on stable reflexive sheaves.

For compact K\"ahler varieties which has quotient singularities,  an orbifold version of Donaldson-Uhlenbeck-Yau theorem was proved by Faulk in \cite{Faulk2022}.    
If the variety $Z$ has quotient singularities only in codimension 2,  
in \cite{Ou2024}, the second author constructed a projective bimeromorphic map $\rho\colon X\to Z$, so that $X$ has quotient singularities, and the indeterminacy locus of $\rho^{-1}$ has codimension at least $3$ in $Z$. 
By using Faulk's theorem on $X$, 
we can then deduce Bogomolov-Gieseker inequalities for orbifold Chern classes of stable reflexive sheaves on $Z$. 

When the variety $Z$ is smooth, the theorem of  Donald-Uhlenbeck-Yau also characterizes the condition when the equality holds in the Bogomolov inequalities.  
This part was not proved in \cite{Ou2024} for singular spaces. 
We focus on this problem in this paper, and  prove the following theorem.  
For the precise definition of the orbifold Chern classes $\hat{c}$, we refer to \cite[Section 9]{Ou2024}. 

\begin{thm}
\label{thm:unitary-flat} 
Let $(Z,\omega)$ be a compact K\"ahler variety of dimension $n$ with klt singularities, 
and let $\mathcal{F}$ be a $\omega$-stable reflexive sheaf on $Z$. 
Then the following two conditions are equivalent. 
\begin{enumerate}
    \item $\hat{c}_2(\mathcal{F})\cdot [\omega]^{n-2} = \hat{c}_1(\mathcal{F})^2  \cdot [\omega]^{n-2} = 0$.
    \item  There is a finite quasi-\'etale cover $p\colon Z'\to Z$, such that the reflexive pullback $(p^*\mathcal{F})^{**}$ is a unitary flat vector bundle. 
\end{enumerate}
\end{thm}

We outline the proof as follows. 
We follow the method of \cite{CGNPPW}, and  combine it with the work of \cite{GPSS}. 
By taking an appropriate bimeromorphic map $\rho\colon X\to Z$ constructed in \cite{Ou2024}, 
we may assume there is an orbifold structure $\mathfrak{X}$ on $X$, so that the pullback $\mathcal{E} := \rho^*\mathcal{F}$ induces an orbifold vector bundle $\mathcal{E}_{\orb}$ on $\mathfrak{X}$. 
Then there is a sequence of orbifold K\"ahler forms $\{\omega_i\}$, which converges to $\rho^*\omega_Z$. 
We may identify each $\omega_i$ as a K\"ahler current on $X$. 
Faulk's theorem implies that there is an orbifold Hermitian-Einstein metric $h_i$  on $\mathcal{E}$ with respect to $\omega_i$.  
If we can prove that such a sequence $\{h_i\}$ converges to some Hermitian-Einstein metric $h_\infty$ with respect to $\rho^*\omega_Z$, then we can deduce Theorem \ref{thm:unitary-flat} by classic curvature calculus.  
We notice that, the Einstein condition is expressed as an elliptic PDE. 
Therefore, if we can obtain certain uniform $L^\infty$ bounds on $\{h_i\}$, 
then we can conclude by  using classic elliptic analysis. 
In order to get such $L^\infty$ bounds, there are two main ingredients. 
The first one is uniform geometric estimates on $(X,\omega_i)$, which was essentially proved in a   series of recent groundbreaking  works: \cite{GPS24},  \cite{GPSS24} and \cite{GPSS}. 
Such estimates are highly non trivial, since the family $\{\omega_i\}$ is degenerating. 
The second main ingredient is essentially proved in the enlightening paper \cite{CGNPPW}. 
Up to renormalizing $h_i$, we can control $\|h_i\|_{L^\infty}$ by $\|h_i\|_{L^1}$,  uniformly in $i$.   
In the end, we can adapt the method of \cite{Simpson1988} to obtain the desired convergence.

There are other approaches to  Bogomolov-Gieseker inequalities and Hermitian-Einstein metrics on singular spaces, see for example \cite{Wu21},   \cite{ChenWentworth2024},  \cite{Chen2025} and \cite{GuenanciaPaun2024}.  
In particular, when $X$ is a klt threefold,  \cite{GuenanciaPaun2024} 
proved  Bogomolov-Gieseker inequalities as well as 
the equality conditions, with a different method. On the other hand, in a very interesting recent preprint \cite{ZZZ2025} (c.f. see also \cite{LZZ17}), the orbifold Bogomolov-Gieseker  inequality is obtained for semi-stable Higgs sheaves on compact K\'ahler varieties with klt singularities. 
It will be interesting to characterize the equality condition there.

The paper is organized as follows. 
In Section \ref{section:preliminaries}, we introduce the notation for the paper, and recall some known results, such as finiteness of fundamental groups for klt singularities, 
and Simpson's operations in his paper \cite{Simpson1988}. 
In Section \ref{section:uniform-kahler}, we establish uniform geometric estimates for a degenerating family of orbifold smooth K\"ahler forms, following the method of \cite{GPSS}. 
In the last section, we establish  mean value type inequalities as in \cite{CGNPPW}, and finish the proof of the main theorem. 
\\

\noindent\textbf{Acknowledgment.} 
The authors are grateful to professor Bin Guo, Jian Song,  Chuanjing Zhang for conversations.  Xin Fu is  supported by National Key R\&D Program of China 2024YFA1014800 and NSFC No. 12401073. Wenhao Ou is supported by the National Key R\&D Program of China (No. 2021YFA1002300).

\section{Preliminaries}\label{section:preliminaries}
We  fix some notation and prove some elementary results in this section.

\subsection{Complex analytic varieties and complex orbifolds}    
A  complex analytic variety  $X$
is a reduced and irreducible  complex analytic space.   
We will denote by $X_{\sm}$ its smooth locus and by $X_{\sing}$ its singular locus.  
A smooth complex analytic variety is also called a complex manifold.    
A complex analytic variety $X$ is said to have klt singularities, 
if for every point $x\in X$, 
there is a neighborhood $U$ of $x$ and a  divisor $\Delta$ on $U$ such that $K_U+\Delta$ is $\mathbb{Q}$-Cartier and $(U,\Delta)$ is a klt pair in the sense of \cite[Definition 2.34]{KollarMori1998}.

We refer to \cite{Grauert1962} for the notion of K\"ahler spaces.    
Assume that $(X,\omega)$ is a  K\"ahler  manifold. 
We will denote by $\Lambda$ the contraction with $\omega$. 
Then the Laplace-Beltrami operator $\Delta$ with respect $\omega$ satisfies 
$\Delta =  2\sqrt{-1}\Lambda\partial\bar{\partial}$.   
We denote 
\[\Delta'=\frac{1}{2}\Delta =  \sqrt{-1}\Lambda\partial\bar{\partial}. \]

Let $X$ be a complex analytic variety of dimension $n$,  
and let $\omega$ be some closed positive $(1,1)$-current with bounded local potentials on $X$. 
Assume that $\omega$ is continuous on some dense Zariski open subset $U\subseteq X$.  
Then we will denote by $\int_{(X,\omega)}$ the integration over $U$ with respect to the volume form $\omega^n$

A complex orbifold $\mathfrak{X}$ with quotient space $X$ is defined by the following data. 
There is an open covering $\{U_i\}$ of $X$, there are complex manifolds $V_i$, there are  finite groups $G_i$ acting holomorphically on $V_i$, such that $V_i/G_i \cong U_i$. 
We require further that the $(V_i,G_i)$ are compatible along the overlaps. 
For more details, we refer to, for example,   \cite[Section 3.1]{DasOu2023}.  
Throughout this paper, we always assume that the actions of $G_i$ are faithful. 
We call the branched locus of the orbifold structure  $\mathfrak{X}$ the subset $Z\subseteq X$, over which the natural morphisms $V_i\to X$ are branched. 
The set $Z$ is always a closed analytic subset of $X$.   
An object $h$, for example a current, a function, etc., is called orbifold smooth, if $h|_{X\setminus Z}$ is smooth, and if the pullback of $h|_{X\setminus Z}$  on any orbifold chart $V_i$ extends to a  smooth object on $V_i$.


A finite morphism $f\colon X\to Y$ between  complex analytic varieties is called quasi-\'etale, if it is surjective and \'etale over an open subset of $Y$, 
whose complement has codimension at least 2. 
The following two  results on finite morphisms   are very useful for this paper. 

\begin{thm}
\label{thm:GR-cover} 
Let $X$ be a complex analytic variety,  
and let $X^\circ \subseteq X$ be a dense Zariski open subset.  
Assume that $X^\circ$ is normal and we have a finite \'etale morphism $p\colon Y^\circ \to X^\circ$. 
Then $p$ extends to a finite morphism $p\colon Y\to X$ with $Y$ normal, which is  unique up to isomorphism.
\end{thm}

\begin{proof} 
Let $r\colon X'\to X$ be the normalization. 
By assumption, $r$ is an isomorphism over $X^\circ$. 
Hence up to replacing $X$ by $X'$, we may assume that $X$ is normal. 
In this case, the theorem is proved in {\cite[Th\'eor\`eme XII.5.4]{SGA1}}.  
\end{proof} 
\begin{lemma}
    \label{lemma:smooth-finite-cover}
Let $Y\subseteq \mathbb{C}^n$ be an open ball centered at the origin, and let $\Delta$ be a divisor on $Y$ which is the union of some coordinate hyperplanes. 
Assume that $\pi\colon W\to Y$ is a finite surjective morphism which is \'etale over $Y\setminus \Delta$. 
Then there is a finite morphism $p\colon Y'\to Y$ such that the following properties hold. 
There is an endomorphism $\rho\colon \mathbb{C}^n\to \mathbb{C}^n$, which can be written in coordinates as 
\[
\rho\colon (z_1,...,z_n) \mapsto (z_1^{a_1},...,z_n^{a_n})  \mbox{ for some integers } a_1,...,a_n>0,  
\]
such that $Y'=\rho^{-1}(Y)$ and $p=\rho|_{Y'}$. 
In addition, the morphism $p$ factors through $\pi$.  
\end{lemma}

\begin{proof}
Without loss of the generality, we may assume that $\Delta$ is defined by  $z_1\cdots z_k=0$ for some integer $1 \le k \le n$.
Let $Y^\circ = Y\setminus \Delta$  and $W^\circ = \pi^{-1}(Y^\circ)$. 
Then the fundamental group   $\pi_1(Y^\circ)$ is isomorphic to $\mathbb{Z}^k$, which is generated by the loops $\gamma_1,...,\gamma_k$ around a general point  the components of $\Delta$.  
Moreover, the morphism $\pi|_{W^\circ}$ corresponds to a subgroup $H$ of  $\pi_1(Y^\circ)$ which has finite index. 
It follows that the subgroup $H'$ generated by $\gamma_1^d,...,\gamma_k^d$ is contained in $H$ for some integer $d>0$ sufficiently large.  

We impose $a_1=\cdots = a_k = d$ and $a_{k+1}=\cdots =a_n=1$ in the definition of $\rho$. 
Let $Y'^\circ = \rho^{-1}(Y^\circ)$. 
Then the natural surjective morphism  $Y'^\circ \to Y^\circ$ is \'etale and corresponds to the subgroup $H'$ of $ \pi^{-1}(Y^\circ)$.  
Hence it factors through $W^\circ$. 
By Theorem \ref{thm:GR-cover}, there is a normal variety   $Y''$ with a finite morphism  $Y''\to W$ which extends $Y^{'\circ} \to W^\circ$.  
It is also clear that the natural finite morphism $Y''\to Y$ extends  $Y^{'\circ} \to Y^\circ$.  
By the uniqueness of Theorem \ref{thm:GR-cover}, we can identify  $Y''$ as   $Y'=\rho^{-1}(Y)$.  
This completes the proof of the lemma.
\end{proof}

We also need the following result on extensions of coherent subsheaves.

\begin{lemma}\label{lemma:extension-subsheaf} 
Let $Z$ be a normal complex analytic variety and $\cF$ a reflexive coherent sheaf. 
Let $Z^\circ \subseteq Z$ be a Zariski open subset whose complement has codimension at least 2, and $\cE \subseteq \cF|_{Z^\circ}$ a saturated coherent subsheaf. 
Then $\cE$ extends to a coherent saturated subsheaf  of $\cF$  on $Z$.
\end{lemma}

\begin{proof}
It is enough to prove the extension locally on $Z$. 
Hence by embedding $\cF$ in a free coherent sheaf as a saturated subsheaf, we may assume that $\cF$ is free. 
Furthermore, by removing from $Z^\circ$ some analytic subset of codimension at least two, 
we may assume that $\cE$ is a subbundle of $\cF|_{Z^\circ}$. 
Then there is an induced  morphism $f\colon Z^\circ \to  M$, where $M=G(n,m)$ is the Grassmannian variety, with $n=\mathrm{rank}\, \cF$ and  $m=\mathrm{rank}\, \cE$.  
By applying \cite[Main Theorem]{Siu1975} to the normal variety $Z$, as explained in \cite[page 441]{Siu1975},  we obtain that $f$ extends to a meromorphic map from $Z$ to $M$. 
Let $\Gamma \subseteq Z\times M$ be the closure of the graph of $f$, 
and we denote   $p_1\colon \Gamma \to Z$ and $p_2\colon \Gamma \to M$.  
Then $p_1$ induces a bimeromorphic map from $\Gamma$ to $Z$.

Let $U$ be the universal vector bundle on $M$. By pulling back to $\Gamma$ \textit{via} $p_2$, we get a subbundle   $\cG$ of $p_1^*\cF$ on $\Gamma$. 
Since   $p_1$ is proper, the direct image $(p_1)_* \cG$ is a coherent  subsheaf of $\cF$ on $X$, extending $\cE$.
This completes the proof of the lemma.
\end{proof}

\subsection{Local and regional fundamental groups of klt singularities}  

Thanks to the Minimal Model Program for projective morphism over a germ of a complex analytic variety, see \cite{DasHaconPaun2022} or \cite{Fujino2022}, 
we can establish the following theorems.   
The first one is on fundamental groups around a klt singularities.  
We recall that the \'etale fundamental group $\pi_1^{\acute{e}t}$ of a topological space is the profinite completion of the fundamental group $\pi_1$. 

\begin{thm}
\label{thm:klt-local}
Let $(x\in X)$ be a germ of  complex analytic variety such that $X$ has klt singularities.  
Then, up to shrinking $X$,  the regional fundamental group $\pi_1^{reg}(X)$ is finite. 
In another word, $\pi_1(X_{\sm})$ is finite. 
\end{thm}

We note that, in the case when the singularity is algebraic, 
it is proved in \cite{Xu2014} that the local \'etale fundamental group $\pi_1^{\acute{e}t}(X\setminus \{x\})$ is finite. 
Later in \cite{Braun2021}, it is shown that the regional fundamental group is finite.

\begin{proof}
In the proofs of \cite[Theorem 1]{Xu2014} and of \cite[Theorem 1]{Braun2021},  
the assumption that the singularity is algebraic is to ensure the existence of plt blowups, 
which extracts    a Koll\'ar component, see \cite[Lemma 1]{Xu2014}. 
Once we get a Koll\'ar component, we can apply the local-global principal to conclude the finiteness theorems.  
The ``local'' part is the fundamental groups of klt singularities, 
and the ``global'' part is the fundamental groups of weakly Fano pairs.  
The tools for  the proof of plt blowups are the theorems in \cite{BCHM10}. 
More precisely, they are  
the existence of MMP  for projective birational morphisms on klt pairs (see \cite[Theorem 1.2]{BCHM10}), 
and the finite generation of log canonical rings (see \cite[Corollary 1.1.2]{BCHM10}). 
In the case of complex analytic varieties, we can apply \cite[Theorem 1.4]{DasHaconPaun2022} or \cite[Theorem 1.7]{Fujino2022} in the place of \cite[Theorem 1.2]{BCHM10}, 
and we can apply \cite[Theorem 1.3]{DasHaconPaun2022} or \cite[Theorem 1.8]{Fujino2022} in the place of \cite[Corollary 1.1.2]{BCHM10}. 
In particular, plt blowups exist on a germ of analytic klt singularity.  
We also note that the ``global'' part   remains the same even we pass to the analytic setting, since the underlying variety of a weakly Fano pair is always a projective variety.  
Hence, once we can extract a Koll\'ar component,  
the same argument in \cite[Theorem 1]{Braun2021} proves the theorem.  
\end{proof}

The following theorem  was proved in \cite[Theorem 1.5]{GrebKebekusPeternell2016b} in the case of projective variety.  
With Theorem \ref{thm:klt-local} in hand, we can adapt its method in the setting of complex analytic varieties.

\begin{thm}
\label{thm:klt-cover}
Let $X$ be a compact complex analytic variety with klt singularities.  
Then there is a finite quasi-\'etale cover $f\colon X' \to X$ such that the following property holds. 
If $\iota\colon X'_{\sm} \to X'$ is the natural inclusion, then the induced morphism  $\iota_* \colon  \pi_1^{\acute{e}t}(X'_{\sm}) \to  \pi_1^{\acute{e}t}(X')$  of \'etale fundamental groups 
is an isomorphism. 
\end{thm}

\begin{proof} 
We first consider a sequence of finite Galois surjective morphisms 
\[
\cdots \to X_k \to \cdots \to X_0 = X,
\]
such that every variety $X_k$ is normal and every  morphism $\varphi_k\colon X_k\to X_{k-1}$ is quasi-\'etale. 
The Zariski's purity theorem then implies that $X_k\to X$ is finite \'etale over $X_{\sm}$. 
We claim that there is  an integer $N> 0$, such that $\varphi_k$ is \'etale over $X_{k-1}$ entirely for $k\geq N$. 
Let $x\in X$ be a point.  
Then by Theorem \ref{thm:klt-local}, there is an open neighborhood $U_x$ of $x$, 
such that  $\pi_1((U_x)_{\sm})$ is finite.  
Then for each $k$, there is a positive integer $m_k$, such that if $V\subseteq X_k$ is a connected component of the preimage of $U_x$, then the natural morphism $V\to U_x$ has degree $m_k$.  We note that $m_k$ is independent of the choice of $V$, since every $\varphi_k$ is Galois.  
It follows that $m_k$ is bounded from above by the order of $\pi_1((U_x)_{\sm})$.   
In addition, $m_{k+1}\ge m_k$ for any $k\ge 0$.  
As a consequence, there is some integer $N(x)>0$, such that if $k\geq N(x)$, then $m_k=m_{k+1}$. 
For such integers $k$, the morphism $\varphi_{k+1}$ is a trivial cover over $(U_x)_{\sm}$.  
It follows that $\varphi_{k+1}$ is a trivial cover over $U_x$. 
By compactness, we can cover $X$ by finitely many such open subsets $U_{x_1},...,U_{x_m}$. 
Let $N =  \max\{ N(x_1),...,N(x_m) \} + 1$. 
Then $\varphi_k$ is \'etale  for $k\geq N$.  

Now we return to the situation of the theorem. 
Since any connected \'etale cover of $X$ induces a connected \'etale cover of $X_{\sm}$, 
  the natural morphism  $\pi_1^{\acute{e}t}(X_{\sm}) \to \pi_1^{\acute{e}t}(X)$ is surjective.  
If it is not an isomorphism, then the kernel of it induces a  Galois finite \'etale morphism $Z \to X_{\sm}$ of degree greater than $1$.  
By Theorem \ref{thm:GR-cover}, it extends to a finite quasi-\'etale morphism $Y\to X$. 
This morphism is not \'etale by construction.   
Hence, if we assume by contradiction that such a finite cover  in the theorem does not exists, 
then by induction, we can construct  an infinite sequence of finite Galois morphisms 
\[
\cdots \to X_k \to \cdots \to X_0 = X,
\]
such that every  $X_k$ is normal and every   $\varphi_k\colon X_k\to X_{k-1}$ is quasi-\'etale but not \'etale.
We obtain a  contradiction to  the first  paragraph. 
This completes the proof of the theorem. 
\end{proof}

\subsection{Simpson's operations}    
Let  $\mathcal{E}$ be a holomorphic vector bundle on a K\"ahler manifold $(X,\omega)$, let $h$ a fixed smooth Hermitian metric on $\mathcal{E}$.  
There is a definite positive Hermitian form on the bundle $\mathcal{E}nd(\mathcal{E}) \cong \mathcal{E}^*\otimes \mathcal{E}$ defined by $\langle A,B\rangle = \mathrm{Tr}(AB^*)$, where $\mathrm{Tr}$ is the trace and $B^*$ is the adjoint of $B$ with respect to $h$.   
Let  $End(\mathcal{E})$ be the set of measurable endomorphism of $\mathcal{E}$, that is,  the set of measurable global section of $\mathcal{E}nd(\mathcal{E})$. 
We denote by $End_h(\mathcal{E}) \subseteq End(\mathcal{E})$   the subset of   self-adjoint  endomorphisms of   with respect to $h$.   
 
Let $\psi: \mathbb R\to \mathbb R$ and 
$\Psi: \mathbb R\times \mathbb R\to \mathbb R$ be two smooth functions. 
They induce maps 
\begin{equation}\label{can1}
\psi: End_h(\mathcal{E})\to End_h(\mathcal{E}), \qquad \Psi: End_h(\mathcal{E})\to End\big(End(\mathcal{E})\big)
\end{equation}
as follows.   
Let $s\in End_h(\mathcal{E})$. 
On a small coordinate subset $U\subset X$, there is an $h$-unitary frame $(e_1,\dots, e_r)$ of $\mathcal{E}$ with respect to which $s$ is diagonal, say $s(e_i)= \lambda_i e_i$ for some real functions $\lambda_i$ defined on $U$.  
Then we set \[\psi(s)(e_i):= \exp(\lambda_i)e_i\] 
for each $i= 1,\dots, r$,  and we obtain in this way a  global Hermitian endomorphism $\psi(s)$.

 Given an $End(\mathcal{E})$-valued  $(p, q)$-form   $A$,  
 we can locally write $\displaystyle A= \sum_{i,j} a^j_i e^i\otimes e_j$, where the coefficients $a^i_j$ are  $(p, q)$-forms on $U$, and  $(e^1,\dots, e^r)$ is the basis on $\mathcal{E}^*$ dual to $(e_1,\dots, e_r)$. 
Then we define
\begin{equation}
\Psi(s)(A)|_U
:=  \sum_{i,j} \Psi(\lambda_i, \lambda_j)a^j_i e^i\otimes e_j,  
\end{equation} 
and particularly $\Psi(s)$  defines globally an endomorphism of $End(\mathcal{E})$, which is self-adjoint with respect to the Hermitian form $\langle \cdot , \cdot  \rangle$.    
By definition, we have the following property.  
\begin{equation}\label{can3}
\mbox{If }  \Psi\geq 0  \mbox{ then }  \langle \Psi (s) A, A\rangle \geq 0  \mbox{ for any }  A\in End(\mathcal{E}).    
\end{equation}

In the following statements, 
for any linear subspace $S$ of $End(\mathcal{E})$,  
we denote by $L^p(S)\subseteq S$ the subspace of elements  which is $L^p$.  
The subspace $L_1^p(S)\subseteq S$ is defined as the subspace of elements $s$ 
such that both $s$ and $\bar\partial s$ are $L^p$.  
For any positive real number $b$, we define $L^p_b(S)$ (respectively $L^p_{1,b})$ the set of elements $s$ in $L^p(S)$ (respectively in $L_1^p(S))$ such that 
$|s|\le b$.



\begin{lemma}\label{lemma:simpsonmor} \cite[Proposition 4.1]{Simpson1988} 
Let $\psi: \mathbb R\to \mathbb R$ and  $\Psi: \mathbb R\times \mathbb R\to \mathbb R $ be two smooth functions. 
Let $b>0$ be a real number. 
Then the following properties  hold. 
\begin{enumerate}

  \item For any $p\ge 1$, there is some $b'>0$ such that the following map is continuous  
    $$\psi: L^p_{b}(End_h(\mathcal{E}))\to L^p_{b'}(End_h(\mathcal{E})). $$

  \item For any $1\leq q\leq p$, we have a  nonlinear map
    $$\Psi: L^p_b(End_h(\mathcal{E}))\to \Hom\big(L^p(End (\mathcal{E})), L^q(End (\mathcal{E}))\big), $$
which is moreover continuous in case when $q< p$.

\item For any $1\leq q\leq p$, there is some $b'>0$ so that we have the following map  
$$\psi: L^p_{1,b}(End_h(\mathcal{E}))\to L^q_{1,b'}(End_h(\mathcal{E})),$$
    which is continuous if  $q<p$. 
    The formula $\bar\partial \psi(s) =   \psi'(s)(\bar\partial s)$ holds in this context.
\end{enumerate} 
\end{lemma}

Now we consider the function $$\Phi (x, y) =\frac{e^{x-y}-1}{x-y}, $$ 
as in \cite{CGNPPW}.  
In the next lemma, we collect some elementary results of $\Phi$ without proof.  

\begin{lemma}\label{lemma:Phi-monotone} 
The following properties hold. 
\begin{enumerate}
\item If $ \alpha< \beta $ are real numbers,  then 
\begin{equation*}
    \Phi(x,y)\geq \frac{\exp(\alpha-\beta )-1}{ \alpha -\beta} \mbox{ for any }
    \alpha \leq x,y \leq \beta.
\end{equation*}  

\item We fix $(x,y)\in \mathbb{R}^2$ and we let 
\[\sigma(\lambda) = \lambda \Psi(\lambda x, \lambda y)  = \frac{\exp(\lambda (x-y)) -1 }{x-y}.\] 
Then $\sigma$ is an increasing function in $\lambda$. 
When $\lambda$ tends to $+\infty$, 
$\sigma(\lambda)$ converges to $\frac{1}{y-x}$ if  $x<y$, 
and tends to $+\infty$ if $x\geq y$.      
\end{enumerate} 
\end{lemma}

\begin{lemma}\label{lemma:A}  
Let $S\in End_h(\mathcal{E})$ be a definite positive smooth global section of $\mathcal{E}nd(\mathcal{E})$. 
Let $s= \log S$, which is also a smooth section in $End_h(\mathcal{E})$.   
We still denote by $\partial$ the $(1,0)$-part of the Chern connection on $(\mathcal{E},h)$. 
Then 
\[\langle(\partial S)S^{-1}, \partial s\rangle_\omega  
=
\langle \Phi(s)(\partial s), \partial s\rangle_\omega. \]
Here the Hermitian form $\langle\cdot , \cdot \rangle_\omega$ on the space of differential forms with values in $\mathcal{E}nd(\mathcal{E})$ is   induced by the K\"ahler form $\omega$ on $X$.   
\end{lemma}

\begin{proof}
They following calculation can be found in \cite[Lemma 2.1]{UhlenbeckYau1986}, see also       \cite[Section 4.2]{CGNPPW}.  
Locally, let $(e_1,...,e_r)$ be a smooth $h$-unitary basis of $\cE$, which diagonalizes $S$, and hence $s$. 
Let $(e^1,...,e^r)$ be the dual basis of $\cE^*$. 
Then we can write $S= \sum \exp({\lambda_i})   {e}^i \otimes e_i$ and $s= \sum \lambda_i  e^i \otimes e_i$, for smooth real-valued functions $\lambda_i$.  
We write $\partial e_i  = A_{i}^j e_j$, where $A_i^j$ are smooth $(1,0)$-forms.  
It follows that $\partial e^i  = - A^{i}_j e^j$. 
Then we have 
\[   
\partial S = \sum \exp({\lambda_i}) \partial \lambda_i  e^i \otimes e_i 
+ \sum (\exp({\lambda_i}) - \exp({\lambda_j})) A_{i}^j e^i \otimes e_j,
\]
and 
\[   
\partial s = \sum  \partial \lambda_i  e^i \otimes e_i + 
\sum ( {\lambda_i} - {\lambda_j}) A_i^j e^i \otimes e_j.
\]
Hence 
\[
\langle(\partial S)S^{-1}, \partial s\rangle_\omega  =  \sum \|\partial \lambda_i\|_\omega^2 
+ 
 \sum (\lambda_i-\lambda_j) ( \exp({\lambda_i}-{\lambda_j}) -1 )\|A_i^j\|^2_\omega,
\]
and 
\[\langle \Phi(s)(\partial s), \partial s\rangle_\omega = \sum \|\partial \lambda_i\|_\omega^2 
+  \sum \frac{\exp(\lambda_i-\lambda_j)-1}{\lambda_i-\lambda_j} \cdot  (\lambda_i-\lambda_j)^2 \|A_i^j\|^2_\omega.
\] 
This completes the proof of the lemma. 
\end{proof}

We observe that the previous constructions are still valid in the setting of  K\"ahler  orbifold.  In the following lemma, we assume that $X$ is the quotient space of a K\"ahler orbifold.

\begin{lemma}\label{lemma:etabound} 
Let $X$ be the quotient space of a compact K\"ahler  orbifold $(\mathfrak{X},\omega_\orb)$, 
and let $(\mathcal{E},h)$ an orbifold Hermitian vector bundle on $\mathfrak{X}$.  
We identify $\omega_\orb$ with a K\"ahler current $\omega$ on $X$ with continuous local potentials. 
Assume that $H$ is a $h$-self-adjoint endomorphism of $\mathcal{E}$ such that $h_{HE}:=h\cdot H$ is a Hermitian-Einstein metric with respect to $\omega_\orb$. 
Let $\eta= \log H$. 
Assume that $\int_{(X,\omega)} \Tr(\eta) =0$. 
Then the following    equality holds, 
\begin{equation}\label{alt5-general}
0 =  
\int_{(X,\omega)}\left\langle \Phi(\eta)(\partial\eta), \partial\eta\right\rangle_\omega 
+ \int_{(X,\omega)} \Tr \big(\eta \circ  \Lambda\Theta_h\big),  
\end{equation} 
where $\Theta_h$ is the Chern curvature tensor of $h$. 
\end{lemma}

\begin{proof}
We can decompose $\eta$ as follows, 
\begin{equation}\label{zz-general}
\eta= -\frac{\rho}{r}\cdot  \ID+ s,
\end{equation} 
where $r$ is the rank of $\mathcal{E}$,  $\Tr s=0$, and $\rho = \Tr \eta$ is an orbifold smooth function on $X$ such that  $\int_{(X,\omega)} \rho =0$. 
The Hermitian-Einstein equation for the metric $h_{HE}$ 
writes
\begin{equation}\label{eqn:HE}
\gamma \cdot \ID- \Lambda\Theta_h= \Lambda \bar\partial\left((\partial H)H^{-1}\right),   
\end{equation}  
where $\gamma$ is an appropriate constant. 
We multiply both side by $\eta$ at the right, and deduce that 
\[
\gamma \cdot \eta =  \Lambda\Theta_h \circ \eta +  \Lambda \bar\partial \left((\partial H)H^{-1}\right) \circ \eta.  
\]
Now we take the trace and integrate on $(X,\omega)$. Since the trace of $\eta$ has mean value $0$, we deduce that 
\[
0 =  \int_{(X,\omega)} \Tr \big(\eta  \Lambda\Theta_h\big)  
+ \int_{(X,\omega)}\Tr \left( \Lambda \bar\partial\big((\partial H)H^{-1}\big)\circ \eta\right)
.\]
For the second summand above,  by integration by part and by noting that $\eta$ is  self-adjoint, we  have  
\[
\int_{(X,\omega)}\Tr \left( \Lambda \bar\partial\big((\partial H)H^{-1}\big)\circ \eta\right)= 
\int_{(X,\omega)} \langle(\partial H)H^{-1}, \partial\eta\rangle_\omega.
\] 
Here, the integration by part  holds, as  all data involved   are orbifold smooth.
It follows that 
\[
0 =  
\int_{(X,\omega)} \langle(\partial H)H^{-1}, \partial\eta\rangle_\omega + \int_{(X,\omega)} \Tr \big(\eta \Lambda\Theta_h\big)
.\]
The lemma  then follows from   Lemma \ref{lemma:A}.
\end{proof}

\section{Uniform estimates on K\"ahler currents}
\label{section:uniform-kahler}



In this section, we will prove some uniform geometric estimates on a degenerate family of orbifold smooth K\"ahler currents. 

\subsection{Uniform estimates on $\mathcal{W}$ classes and $\mathcal{AK}$ classes. }
We start by recalling the recent breakthrough   by \cite{GPSS} and \cite{GPSS24},  
on uniform geometric estimates of a very robust family of K\"ahler metrics.   
Firstly, we record a series of definitions from these two papers.

\begin{defn}
\label{def:Wclass}
Let $(Y,\theta_Y)$ be a compact K\"ahler manifold of dimension $n$. 
Let $p\geq 1$, $A,K>0$ be real numbers and let $\gamma$ be a non-negative continuous function  on $Y$.   
We say that a K\"ahler form $\omega$ belongs to the class $\mathcal{W}(Y, \theta_Y, n, p, A, K, \gamma)$ if the following properties hold. 
\begin{enumerate}
    \item $ [\omega] \cdot  [\theta_Y]^{n-1} \le A$. 
    \item The $p$-th Nash-Yau entropy is bounded by $K$, i.e.
$${\mathcal{N}}_p (\omega) = \frac{1}{V_\omega} \int_Y \left|\log \frac{1}{V_\omega} \frac{\omega^n}{\theta_Y^n} \right|^p \omega^n  \leq K, $$ 
where $V_\omega=  [\omega]^n$ is the volume of $(X,\omega)$.  
\item  $\frac{\omega^n}{\theta_Y^n} \geq \gamma$. 
\end{enumerate}

\end{defn}

\begin{defn} \label{def:akclass} Let $X$ be a  compact normal K\"ahler variety of dimension $n$, let $\pi: Y \rightarrow X$ be a log resolution of singularities and let $\theta_Y$ be a smooth K\"ahler form  on   $Y$.   
We fix  constants $A,K>0$, an integer $p>n$,  and a non-negative function  $\gamma \in \mathcal{C}^0(Y)$ such that $\{y\in Y \ | \ \gamma(y)=0\}$ is contained in a proper analytic subvariety of $Y$. 
Then   the set of admissible  K\"ahler currents 
$${\mathcal{AK}}(X, \theta_Y, n, p, A, K, \gamma) $$  
is defined to be the set of   K\"ahler currents $\omega$ on $X$ satisfying the following conditions.

\begin{enumerate}

\item $[\omega]$ is a K\"ahler class on $X$ and $\omega$ has bounded local potentials.

\item $[\pi^*\omega] \cdot [\theta_Y]^{n-1}\leq A$ and $[\omega]^n \geq A^{-1}$.

\item The $p$-th Nash-Yau entropy is bounded by $K$, i.e.
$${\mathcal{N}}_p (\omega) = \frac{1}{V_\omega} \int_Y \left|\log \frac{1}{V_\omega} \frac{(\pi^*\omega)^n}{\theta_Y^n} \right|^p (\pi^*\omega)^n  \leq K, $$
where $V_\omega=  [\omega]^n$.

\item 
 $\frac{(\pi^*\omega)^n}{ \theta_Y^n}\geq \gamma.$
 
\item The log volume measure ratio 
$$\log \left( \frac{(\pi^*\omega)^n}{\theta_Y^n} \right)$$ has log type analytic singularities.
\end{enumerate} 
\end{defn}

The log type analytic singularities in the item $(5)$ above is defined as follows.

\begin{defn}\label{def:logtype}
A  function $F$ on a  complex $Y$ manifold of dimension $n$ is said to have log type analytic singularities if the following properties hold. 
There exist 
smooth prime divisors $D_1,...,D_N$ on $Y$ with simple normal
crossings. 
For $j = 1, ...., N$,   
let $\sigma_j \in H^0(Y,\mathcal{O}_Y(D_j))$ be a defining section of $D_j$, 
and $h_j$ be a smooth hermitian metric on  $\mathcal{O}_Y(D_j)$. 
Locally around every point of $Y$, the function $F$ can be written in the shape
$$F =\sum_{k=1}^K a_k(-\log)^k\Big(\prod
^N
_{j=1}
e^
{f_{k,j}}|\sigma_j|^{
2b_{k,j}}_
{h_j}\Big),$$
where $K\geq 1$ is an integer, 
$(-\log)^k$
is the $k$- the composition of $(-\log)$, 
$a_k, b_{k,j} \in \mathbb{R}$  and $f_{k,j} \in C
^{\infty}(Y)$. 
\end{defn}

With the notation of Definition of \ref{def:akclass},  for any $\omega \in \mathcal{AK}(X, \theta_Y, n, p, A, K, \gamma)$, we  follow \cite{GPSS} to set $$\cS_{X, \omega}:=X_{\mathrm{sing}}\cup \,\,\pi \left( \textnormal{Singular set of}\,\,\, \left( \log \frac{(\pi^*\omega)^n}{\theta_Y^n} \right) \right).$$   From the definition of log type singularities, we see that $\cS_{X, \omega}$ is an analytic subvariety of $X$. 

\begin{defn}\label{def:metric-completion} 
With the notation of Definition \ref{def:akclass},   assume that 
$$\omega \in \mathcal{AK}(X, \theta_Y, n, p, A, K, \gamma). $$ 
We define
$$(\hat X, d) = \overline{(X\setminus \cS_{X, \omega}, \omega|_{X\setminus \cS_{X, \omega}})}$$
to be the metric completion of $\left(X\setminus \cS_{X,\omega}, \omega|_{X\setminus \cS_{X,\omega}} \right)$. We also denote the unique metric measure space  associated to $(X, \omega)$ by 
$$(\hat X, d, \omega^n).$$

\end{defn}

We remark that $\omega^n$ extends uniquely to a volume measure on $\hat X$ because neither $\omega$ nor $\omega^n$  carries mass on $\cS_{X, \omega}$. 
We can consider the Sobolev space $L_1^{2}(\hat X,d, \omega^n)= W^{1,2}(\hat X,d, \omega^n)$ as in \cite[Definition 8.1]{GPSS}. 


%
%

 Guo-Phong-Sturm-Song prove a  package of uniform geometric estimates.  
\begin{thm}\cite[Theorem 3.1]{GPSS}\label{thm:soborbi} Let $\omega\in\mathcal{AK}(X, \theta_Y, n, p, A, K, \gamma)$, then the following properties hold:
\begin{enumerate}

\item There exists a constant $C=C(X, \theta_Y, n, p, A, K, \gamma)>0$ such that  
$${\textnormal{diam}}(\hat X, d) \leq C.$$
In particular, $(\hat X, d)$ is a compact metric space.
 
\item  There exist a constant $q>1$  and a constant  $C_S=C_S(X, \theta_Y, n, p,  A, K, \gamma, q)>0$ such that the following Sobolev inequality  
$$
\Big(\int_{\hat X} | u  |^{2q}\omega^n   \Big)^{1/q}\le C_S \left( \int_{\hat X} |\nabla u|^2 ~\omega^n + \int_{\hat X} u^2 \omega^n \right)  
$$
holds for all $u\in W^{1, 2}(\hat X, d, \omega^n)$.

\item There exists a constant $C_H=C_H(X, \theta_Y, n, p, A, K, \gamma,q)>0$ such that the following trace formula holds for the heat kernel $H$ of $(\hat X, d, \omega^n)$, 
$$H(x,x, t) \leq \frac{1}{V_\omega} + \frac{C_H}{V_\omega} t^{-\frac{q}{q-1}}. $$

\item Let $0=\lambda_0 < \lambda_1 \leq \lambda_2 \leq ... $ be the increasing sequence of eigenvalues of the Laplacian $-\Delta_\omega$ on $(\hat X, d, \omega^n)$. Then there exists $c=c(X, \theta_Y, n, p, A, K, \gamma, q)>0$ such that
$$\lambda_k \geq c k^{\frac{q-1}{q}}. $$
\end{enumerate}
\end{thm}

We recall the definition of  heat kernels.  
\begin{defn}\label{def:Heatkernel} 
Assume that $  \omega\in\mathcal{AK}(X, \theta_Y, n, p, A, K, \gamma)$.  
The heat kernel of the Laplacian $\Delta_{\omega}$ is by the following parabolic equations 
$$\partial_{t} H(x, y, t) = \Delta_{\omega, y} H(x, y, t),  \ \lim_{t\rightarrow 0^+} H(x, y, t) = \delta_x(y)$$
for $x, y \in Y^\circ = \pi^{-1}(X\setminus \cS_{X, \omega})$. 
\end{defn}

We will need the following uniform mean value inequality, which is essentially proved in \cite[Lemma 2]{GPS24} and \cite[Lemma 5.1]{GPSS24}. 

\begin{lemma}\label{lemma:meanvalue} 
Let $\omega\in \mathcal{W}(Y, \theta_Y, n, p, A, K, \gamma)$. 
We assume in addition there is some constant $B>1$ such that $V_{\omega} = [\omega]^n$ is contained in $[B^{-1},B]$.  
Let $a$  and $I$ be positive real numbers. 
Let  $v\in L^1(Y, \omega)$ be a function such that $ |\int_{(Y,\omega)}v| \le I$. 
Assume that $v$ is  $\mathcal{C}^2$-differentiable on the set 
$\{v>  -I\cdot B  -1\}$ and  that  
\[\Delta_{\omega}(v)\geq -a\]
on $\{v>  -I\cdot B  \}$. 
Then we have 
\[ v\leq C\big(1+ \Vert v\Vert_{L^1(Y,\omega)}\big)
\] 
where $C=C(Y,\theta_Y, n, p , A,K,\gamma,a,B,I)$ is a positive real number. 
\end{lemma}

\begin{proof}
Let $M = \frac{1}{V_\omega} \int_Y v\cdot \omega^n$ and let $u=v-M$. 
Then $\int_Y u\cdot \omega^n =0$ and $|M|\leq IB$.  
Thus $u$ is $\mathcal{C}^2$-differentiable on the set 
$\{u>    -1\}$ and  that  
$\Delta_{\omega}(u)\geq - a$ 
on $\{u>  0  \}$.   
By \cite[Lemma 5.1]{GPSS24}, there is a constant $C'=C'(Y,\theta_Y, n,p,A,K,\gamma,a)$ such that 
$u \leq C'(1+ \| u \|_{L^1(Y,\omega)})$. 
We note that $   \| u \|_{L^1(Y,\omega)} \leq  \| v \|_{L^1(Y,\omega)} + |M| \cdot V_\omega$. 
Hence we have  
\[
v \le IB + C'(1 +  \| v \|_{L^1(Y,\omega))} + IB^2) 
\]
This completes the proof of the lemma. 
\end{proof}

\subsection{Uniform estimates for degenerating families of orbifold K\"ahler forms}

In the remainder of this section, we consider the following situation. 
Let $(Z,\omega_Z)$ be a compact K\"ahler variety of dimension $n$. 
Assume that $\rho\colon X\to Z$ and $\pi\colon Y\to X$ are projective bimeromorphic morphisms, such that $Y$ is smooth and $X$ is the quotient space of some K\"ahler  orbifold $\mathfrak{X}$.  
Let $\theta_Y$ be a K\"ahler form on $Y$.  
We assume that the $\rho\circ \pi$-exceptional locus is a snc divisor, 
and that there is a divisor $D\geq 0$ with the same support, such that $-D$ is relatively ample over $Z$.    
In particular, we fix a smooth Hermitian metric $h_D$ on $\mathcal{O}_Y(D)$ so that $(\rho\circ\pi)^*\omega_Z - \delta \cdot \Theta_{h_D}$ is a K\"ahler form for all $\delta>0$ small enough, where $\Theta$ stands for the Chern curvature.  
We assume further that $\pi(D)$ contains the branched locus of $\mathfrak{X}$, 
denote by $s_D\in H^0(Y,\mathcal{O}_Y(D))$ a global section defining $D$. 
We also suppose that $\pi(D)$ contains the branched locus of $\mathfrak{X}$.     
Let $\omega_\orb$ be an orbifold K\"ahler form on $\mathfrak{X}$. 
By abuse of notation, we also denote by $\omega_\orb$ the induced K\"ahler current on $X$. 
For any $\epsilon>0$, we set $\omega_\epsilon = \rho^*\omega_Z + \epsilon\cdot \omega_\orb$. 
They are considered as K\"ahler currents on $X$, which are orbifold smooth on $\mathfrak{X}$.  
Our objective is to show   the following theorem. 

\begin{thm}
\label{thm-orbifold-AK-property} 
There exists constants $C,C_S,C_H,c>0$, all independent of $\epsilon$, 
such that for all $\epsilon>0$ small enough, 
the consequences of Theorem \ref{thm:soborbi} hold for $\omega_\epsilon$. 
\end{thm}
We follow the method of \cite[Section 7]{GPSS}, and will approximate  $\omega_\epsilon$ by certain family  $\{\omega_{j}\}_{j\gg 0}$ of  smooth K\"ahler forms on $Y$.  
The key is to show that, for all $j$ sufficiently large, 
$\omega_j$ belong to the same class $  \mathcal{W}(Y, \theta_Y, n, p, A, K, \gamma)$, where $A,K,p,\gamma$ are independent of $\epsilon$ and $j$.    
For more details, see Lemma \ref{lemma:unisob}.

In our situation, it is routine to verify that the currents $\omega_\epsilon$ satisfy the items $(1)$-$(4)$ of Definition \ref{def:logtype},  uniformly for all  $\epsilon>0$ small enough.  
However, for the item $(5)$, the log volume ratio $\log \left(\frac{\pi^*\omega^n_{\epsilon}}{\theta_Y}\right)$  does not have log type analytic singularities.
Fortunately, in the following lemma, we observe that the log volume ratio is the sum of two functions, one has log type analytic singularities, and the other one is a continuous function $G_\epsilon$.  
For this function $G_\epsilon$, we can estimate the blow-up rates of its derivatives with respect to the distance to the divisor of $D$. 
With a smoothing argument by using convolutions, we can still find the desired approximations.


\begin{lemma} \label{l:logtype} 
There   are effective divisor $E_1,E_2$ without common components, 
and a constant  $a>0$ such that the following properties hold. 
The  supports of $E_1$ and $E_2$  are contained in the one of $D$.  
For $i=1,2$, let $s_{E_i}\in H^0(Y,\mathcal{O}_Y(E_i))$ be a global  section  defining $E_i$, 
let $h_{E_i}$ be a smooth Hermitian metric on $E_i$. 
We set  
\[
F_{\log} = a (\log |s_{E_1}|_{h_{E_1}} -  \log |s_{E_2}|_{h_{E_2}}). 
\]
For any   K\"ahler current $\omega$ on $X$ which is an orbifold K\"ahler form on $\mathfrak{X}$,   
we can write  
\[\log\frac{\pi^*\omega^n}{\theta_Y^n}=F_{\log} +  G\] 
where $G$ is  a continuous  function on $Y$, smooth away from $D$.

In addition, 
for any integer $k>0$, there  are positive integers $C_k, N_k$ such that
$\|\nabla^k G\|_{\theta_Y}\leq C_k |s_D|_{h_D}^{-N_k}$, where $\nabla$ means the covariant derivatives with respect to $\theta_Y$.
\end{lemma}

\begin{proof}  
We investigates the  first part of the lemma   locally on $X$.  
Assume that  $f: V\to X$ is an orbifold chart.  
We denote by $\omega_V=f^*\omega $ the smooth K\"ahler form on $V$. 
Let 
$W$ be the normalization of $V\times_X Y$.   

Then we study locally on $Y$. 
By abuse of notation,   we will assume that $Y\subseteq \mathbb{C}^n$ is an open ball centered at the origin.  
Since $D$ is snc, 
we can  assume that  it is a union of coordinates hyperplanes. 
Since the branched locus of $W\to Y$ is contained in the divisor $D$, 
by applying  Lemma \ref{lemma:smooth-finite-cover} to the finite cover $W\to Y$,  we obtain a finite morphism  $p\colon Y'\to Y$.  
We may also assume that $\theta_Y$ is equal to the Euclidean K\"ahler form. 
Let $\theta_{Y'}$ be the Euclidean K\"ahler form on $Y'$.  
Then, by the construction of Lemma \ref{lemma:smooth-finite-cover},  
we have  
\[(p^*\theta_Y)^n = (\prod_{i=1}^n  a_i^2 |z_i|^{2a_i-2}) \cdot \theta_{Y'}^n , \]  
where$(z_1,...,z_n)$ are coordinates on $Y'$  and $a_1,...,a_n$ are positive integers.

Now we compute the ratio $\frac{\pi^*\omega^n}{\theta_Y^n}$ by pulling it back to $Y'$. 
Up to shrinking $V$, we may assume that $\omega_V^n= \rho\cdot \Theta \wedge \overline{\Theta}$, where $\Theta$ is a nowhere vanishing holomorphic $n$-form independent of $\omega_V$,  and $\rho$ is a smooth nowhere vanishing function on $V$.  
If $q\colon Y'\to V$ is the natural morphism, then the locus where $q$ is not smooth is contained in $p^{-1}(D)$, which is a union of coordinate hyperplanes.  
It follows that $q^*\Theta$ is a holomorphic $n$-form on $Y'$,  whose vanishing locus is contained in $p^{-1}(D)$. 
Hence we can write  
\[
q^*(\Theta\wedge \overline{\Theta}) = A\cdot \varphi_2 \cdot \theta_{Y'}^n
\]
where $A$ is a positive smooth function, and  $\varphi_2$ is of the shape 
\[
\varphi_2 = \prod_{i=1}^n   |z_i|^{2b_i-2} 
\]
for some  positive integers $b_1,...,b_n$.  
Hence we can write 
\[
(q^*\omega_V)^n =   \varphi_1 \cdot \varphi_2 \cdot   \theta_{Y'}^n, 
\] 
where $\varphi_1$ is a smooth positive function.  
Therefore, we can write 
\[
p^*(\frac{\pi^*\omega^n}{\theta_Y^n}) = \frac{(q^*\omega_V)^n}{\theta_{Y'}^n} \cdot \frac{\theta_{Y'}^n}{(p^*\theta_Y)^n} = \psi_1\cdot \psi_2, 
\]
where $\psi_1=\varphi_1$ is a smooth positive function,
and 
\[\psi_2 := \varphi_2\cdot \prod_{i=1}^n |z_i|^{2-2a_i} = \prod_{i=1}^n   |z_i|^{2c_i},  \]  
for some  integers $c_1,..,c_n$.  
We remark that the product $\psi_1\cdot \psi_2$ is invariant under the Galois group of $Y'\to Y$, and so is $\psi_2$. 
Thus, so is $\psi_1$. 
Hence there is a continuous positive function $\eta_1$ on $Y$ whose pullback on $Y'$ is equal to $\psi_1$.   
Similarly, $\psi_2$ descend to a function  $\eta_2$ on $Y$, which has the shape $\eta_2=\prod_{i=1}^n |t_i|^{d_i}$, for some rational numbers $d_1,...,d_n$, 
where $(t_1,...,t_n)$ are coordinates on $Y$.  

It follows that the singularities of the log volume ratio $\log\frac{\pi^*\omega^n}{\theta_Y^n}$ is identical  to those of  $\log \eta_2 = \log (\prod_{i=1}^n |t_i|^{d_i})$.  
Furthermore, from the construction, $\eta_2$ may depend on $Y'$ and $V$, but is independent of $\omega$. 
Hence the $\mathbb{Q}$-divisors locally defined by $\prod_{i=1}^n |t_i|^{d_i}=0$ glue globally into a $\mathbb{Q}$-divisor  $\Delta$,  which depends only on $X$ and $Y$.  
There is a positive integer $m$, such that $m\Delta$ is integral. 
We define $E_1$ and $E_2$ so that  $m\Delta=E_1-E_2$, and define $a= m^{-1}$. 
Then the function 
\[
G  = \log\frac{\pi^*\omega^n}{\theta_Y^n} - F_{\log}
\]
is  continuous on $Y$. 
In addition,  $p^*G$ is smooth on $Y'$. 
This proves the first statement of the lemma. 

For the second part of the lemma, since $Y$ is compact, we only need to prove the estimate locally on $Y$. 
Hence we can still use the previous notation and consider $Y$ as an open ball in $\mathbb{C}^n$.   
We fix some integer $k\geq 0$. 
By pulling back to $Y'$, we see that 
\[
p^*(\nabla^k G)  = \nabla^k(G_1+\log \psi_1). 
\]
for some smooth function $G_1$ on $Y$. 
In particular, $\|p^*(\nabla^k G) \|_{\theta_{Y'}}$ is bounded.  
Without loss of the generality, we can  assume the support of $D$ is defined by  $t_1\cdots t_j =0$ for some integer $  j\le n$. 
Then  $\|\nabla^k G \|_{\theta_{Y}}$ is bounded by $C'_k\cdot |t_1\cdots t_j|^{-m_k}$ for some positive integers $C'_k,m_k$.  
By our assumption on $s_D$, we see that $|s_D|_{h_D}$ can be written as $C''\cdot |t_1|^\alpha_1 \cdots |t_j|^{\alpha_j}$ for some positive integers $\alpha_1,...,\alpha_j$ and some positive smooth function $C''$.  
Hence  $\|\nabla^k G\|_{\theta_Y}\leq C_k |s_D|_{h_D}^{-N_k}$ for some positive integers $C_k,N_k$. 
This completes the proof of the lemma.  
\end{proof}

In the previous lemma, both $E_1$ and $E_2$ are allow to be the zero divisor. 
We will later use convolutions to  approximate the function $G$ above by smooth functions.  
The following proposition provides some estimates on the convolutions.




\begin{prop}\label{prop:appg} Let $G$ be a continuous function on  $Y$ 
which is smooth away from $D$. 
Assume that for any integer $0\leq k\leq 3$,  
there are integers $C'_k >0$ and $a_k\geq 0$, 
such that all covariant derivatives  of $G$ up to order $k$ with respect to $\theta_Y$, 
are bounded by $C'_k\cdot |s_D|_{h_D}^{-a_k}$.   
Then there exists  a family of smooth approximating functions $\{G_\sigma\}_{1 \gg \sigma>0}$ of $ G$ satisfying the following properties:
\begin{enumerate}
\item $ G_\sigma   $ are bounded,  uniformly for all  $\sigma$.  
\item On any compact set $K\subset Y\setminus D$, $G_\sigma$ converge to $G$ uniformly and smoothly as $\sigma\rightarrow 0$.
\item There are  positive integers $C, d$, 
such that $\|\nabla^2 G_\sigma\|_{\theta_Y} \leq C|s_D|_{h_D}^{-d}$ for all $\sigma$ small enough. 
\end{enumerate}
\end{prop}

\begin{proof} 
Let $\theta_{1}\colon \mathbb{R} \to \mathbb{R}_{\geq 0}$ be a   function supported on $[0,1]$, such that $\theta_1(|w|^2)$ is smooth for  $w\in \mathbb{R}^{2n}$  and that $\int_{\mathbb{R}^{2n}}\theta_1(|w|^2) \mathrm{d}w = 1$.   
For $\sigma>0$, we set  $\theta_{\sigma}(u) = \frac{1}{\sigma^{2n}}\theta_1(\frac{u}{\sigma^2})$, 
so that it is supported on $[0,\sigma^2]$ and $\int_{\mathbb{R}^{2n}}\theta_\sigma(|w|^2) \mathrm{d}w = 1$.    
We will use the functions $\theta_{\sigma}$  as convolution kernels to construct approximations of $G$.  
We note that there are  integers $C_0, b_0>0$, such that the derivatives of $\theta_\sigma$ up to order $2$ is bounded by $ C_0 \cdot \sigma^{-b_0}$.

We  denote  by $|x-y|$ the distance between two points $x,y \in Y$. 
Since $Y$ is compact, there is some $0<\sigma_0<1$ small enough, such that for any $y\in Y$, 
the exponential map $\exp_y$ is an isomorphism from the ball in $\mathbb{R}^{2n}$ of radius $4\sigma_0$ centered at the origin.   
From now on, we only consider $\sigma>0$ which are less than $\sigma_0$,   
and define the  smooth functions  $G_\sigma$ by using convolutions as follows, 
\[G_\sigma (y) = \int_{w\in \mathbb{R}^{2n}}  \theta_\sigma (|w|^2) \cdot G(\exp_y(w)) \mathrm{d}w.  \] 
By our choice of $\sigma_0$, we have the following alternative expression of $G_\sigma$, 
 \[G_\sigma (y) =  \int_{x\in (Y,\theta_Y)} \theta_\sigma (|x-y|^2) \cdot G(x) \cdot \lambda(y,x),  \]
where $\lambda(y,x)^{-1}$ is the Jacobian determinant of the exponential map $\exp_y$ at the point $(\exp_y)^{-1}(x)$. 
Up to shrinking $\sigma_0$, we can assume that $\exp_y^{-1}$ and  $\lambda(y,x)$ are  smooth function on $\{(x,y)\in Y \times Y \ | \ |x-y|<4\sigma_0 \}$. 
From the standard properties of convolutions, we can deduce the items $(1)$ and $(2)$.

For the item $(3)$, 
we  set \[T_\sigma = \{x\in Y \ | \  \mathrm{dist}\, (x, D) < \sigma\},\] 
where $\mathrm{dist}\, (x, D)$ is the distance from $x$ to $D$. 
Locally around every point of $D$, there is a coordinate neighborhood with holomorphic  coordinates $(z_1,...,z_n)$, on which $D$ is the union of certain coordinate hyperplanes. 
In particular, $|s_D|_{h_D}$ can be written in the shape 
\[ |s_D|_{h_D} =  A\cdot |z_1|^{\alpha_1}\cdots |z_n|^{\alpha_n}\] 
for some smooth positive function  $A$ and for some $\alpha_1,...,\alpha_n \in \mathbb{Z}_{\geq 0}$. 
Therefore, since $Y$ is compact, there are positive constant integers  $C_1, b_1$, 
such that for all $0<\sigma<\sigma_0$,  we have 
\begin{equation} \label{eqn:appg1}
\sigma \geq \frac{1}{2} \mathrm{dist}(x,D)  \geq  C_1 |s_D(x)|_{h_D} ^{b_1}  \mbox{ for all } x\in T_{2\sigma}.      
\end{equation}
In addition, there is a constant $C_2$, such that for any $y\in  Y\setminus T_{2\sigma}$, and for any $t\in Y$ with $|t-y|\le \sigma$, 
we have  
\begin{equation} \label{eqn:appg2}
    |s_D(t)|_{h_D} \geq C_2 |s_D(y)|_{h_D}. 
\end{equation}
To visualize this constant $C_2$, locally around a point of $D$ for example, 
we may let $C_2 = 2^{-(\alpha_1+\cdots + \alpha_n)}$

We fix an open covering of $Y$ by coordinates open subsets.  
It is enough to prove that, 
there are constant integers $C> 0$ and $d \geq 0$, such that  on each of these open subsets, we have 
\[
|\partial_{z_i} \partial_{\bar{z}_j} G_\sigma| \le C\cdot|s_D|_{h_D}^{-d}, 
\] 
for all $i,j$ and all $\sigma$. 
The idea to divide the manifold $Y$ into two parts (depending on $\sigma)$ and estimate the derivatives of $G_\sigma$  separately. 


Firstly, we assume that  $y \in Y\setminus T_{2\sigma}$.    
Then for any $w\in \mathbb{R}^{2n}$ with $|w|\le \sigma$,  
the partial derivatives with respect to $y$ satisfies 
\[
|\partial_{z_i} \partial_{\bar{z}_j}G(\exp_y(w))-\partial_{z_i} \partial_{\bar{z}_j}G(\exp_y(0))|  
\le \sigma \cdot |\varphi(y,w')| \le |\varphi(y,w')|, 
\]
where $\varphi$ involves the   partial derivatives of $\partial_{z_i} \partial_{\bar{z}_j}G(\exp_y(w))$ with  respect to $w$, and $w'$ is a point lying on the interval $[0,w]$ inside $\mathbb{R}^{2n}$.  
Since $\exp_y(w)$ is a smooth function for $y\in Y$ and $|w|<4\sigma_0$ by our choice of $\sigma_0$, its partial derivatives up to order $3$ are bounded by a constant, whenever $|w|\le \sigma_0$. 
Therefore, by chain rule, if we set $t=\exp_y(w')$, then the term $|\varphi(y,w')|$ can be controlled by the partial derivatives of $G$ up to order $3$ at the point $t$.  
From the estimates on the partial derivatives of $G$, 
we then deduce that 
\[
|\partial_{z_i} \partial_{\bar{z}_j}G(\exp_y(w))-\partial_{z_i} \partial_{\bar{z}_j}G(\exp_y(0))|  \le |\varphi(y,w')| 
\le C_3\cdot C_3'  |s_D(t)|_{h_D}^{-a_3}, 
\] 
for some constant $C_3$. 
Since $y\in  Y\setminus T_{2\sigma}$, we have  $|s_D(t)|_{h_D} \geq C_2 |s_D(y)|_{h_D}$, see \eqref{eqn:appg2}.  
Thus  
\begin{eqnarray*}
|\partial_{z_i} \partial_{\bar{z}_j} (G_\sigma - G) (y) | 
&\le&  \int_{w\in \mathbb{R}^{2n}}  \theta_\sigma (|w|^2) \cdot |\partial_{z_i} \partial_{\bar{z}_j} G(\exp_y(w)) - \partial_{z_i} \partial_{\bar{z}_j}G(\exp_y(0))| \mathrm{d}w  \\
&\le&   C_3 C_3'  |s_D(t)|_{h_D}^{-a_3}  \int_{w\in \mathbb{R}^{2n}}  \theta_\sigma (|w|^2) \cdot  \mathrm{d}w   \\
&\le&   C_3 C_3'  \cdot  (C_2|s_D(y)|_{h_D})^{-a_3}. 
\end{eqnarray*}
Hence $\partial_{z_i} \partial_{\bar{z}_j} G_\sigma (y)$ is bounded by $C'_2 |s_D(y)|_{h_D}^{-a_2} + C_3C'_3 \cdot  (C_2|s_D(y)|_{h_D})^{-a_3}$.  

Assume that $y\in T_{2\sigma}$.  
Since $G$ is continuous, we may assume that $|G| $ is bounded by the constant $ C_0'$. 
Then, by considering the partial derivatives with respect to $y$, we have 
\begin{eqnarray*}
|\partial_{z_i} \partial_{\bar{z}_j} G_\sigma   (y)|  
&=& \left|\int_{x\in (Y,\theta_Y)} (\partial_{z_i} \partial_{\bar{z}_j} (\theta_\sigma (|y-\cdot|^2 )\cdot \lambda(y,\cdot))(x) )\cdot G(x) \right| \\
&\le & \int_{ x\in (Y,\theta_Y)}  C_4 \cdot C_0 \cdot \sigma^{-b_0} \cdot |G(x)|  \\
&\le & \int_{x\in (Y,\theta_Y)}  C_4 \cdot C_0 \cdot \sigma^{-b_0} \cdot C'_0,  
\end{eqnarray*} 
where $C_4$ is a constant independent of $\sigma$ and $y$. 
For the first inequality above, we use the estimates on the derivatives of $\theta_\sigma$ to obtain the term $C_0 \sigma^{-b_0}$. 
We also use the fact that the partial derivatives, with respect to $y$ up to order $2$, of $|y-x|^2$ and of  $\lambda(y,x)$,  are bounded by some constant, over the domain $\{(x,y)\in Y\times Y \ | \ |x-y|<\sigma_0 \}$. 
Since  $y\in T_{2\sigma}$, as shown in \eqref{eqn:appg1},  we have 
\[
\sigma^{-b_0}  \le (C_1 |s_D(y)|_{h_D}^{b_1})^{-b_0}. 
\]
Hence $\partial_{z_i} \partial_{\bar{z}_j} G_\sigma   (y)$ is bounded by 
\[C_4\cdot C_0 \cdot C_1^{-b_0}  \cdot C_0' \cdot {Vol}(Y,\theta_Y) \cdot |s_D(y)|_{h_D}^{-b_1b_0},\] 
where $Vol$ is the volume. 
This completes the proof of the proposition. 
\end{proof}




In the following proposition, we prove that the family of currents $\{\omega_{\epsilon}\}$  satisfies  certain uniform   estimates. 

\begin{prop}\label{prop:VerifyAKclass}  
There exists $A,K,p,\gamma$ such that  $\omega_\epsilon$ satisfies the assumption $(1)$-$(4)$ of Definition \ref{def:akclass} for all $0<\epsilon \leq 1$. 
In other words, the following properties hold. 
\begin{enumerate}

\item  ${\omega_\epsilon}$ has bounded local potentials.

\item $[\pi^*{\omega_\epsilon}] \cdot [\theta_Y]^{n-1}\leq A$ and $[{ \omega_\epsilon}]^n \geq A^{-1}$.

\item The $p$-th Nash-Yau entropy is bounded by $K$, i.e.
$${\mathcal{N}}_p ({\omega_\epsilon}) = \frac{1}{V_{\omega_\epsilon}} \int_Y \left|\log \frac{1}{V_{\omega_\epsilon}} \frac{(\pi^*{\omega_\epsilon})^n}{\theta_Y^n} \right|^p (\pi^*{\omega_\epsilon})^n  \leq K, $$
where $V_{\omega_\epsilon}=  [{\omega_\epsilon}]^n$.

\item 
 $\frac{(\pi^*{\omega_\epsilon})^n}{ \theta_Y^n}\geq \gamma.$

\end{enumerate} 
\end{prop}

\begin{proof} 
The item $(1)$ holds, since the local potentials of $\omega_\epsilon$ are orbifold smooth, and hence bounded.  
The  item $(2)$ follows from the the monotonicity of $\omega_\epsilon$ in $\epsilon$ and  the fact that  $\omega_Z$ is K\"ahler.    
For the item $(4)$, by the monotonicity of $\omega_\epsilon$ again, it is enough to set 
\begin{equation}\label{eqn:gamma}
\gamma = \frac{(\rho\circ \pi)^*\omega_Z^n}{\theta_Y^n}. 
\end{equation}

 
It remains to prove the item $(3)$ on Nash-Yau entropies.  
We notice that $\frac{\pi^*{\omega_1^n}}{\theta_Y^n}$ is integrable on $(Y,\theta_Y)$,  since $\omega_1$ is orbifold smooth.  
From Lemma \ref{l:logtype}, we see that  $\frac{\pi^*{\omega_1^n}}{\theta_Y^n}$ has polynomial  poles (with rational exponents)  along $D$. 
It follows that  $\frac{\pi^*{\omega_1^n}}{\theta_Y^n}$ is $L^{1+\delta_0}$ integrable for some $\delta_0>0$. 
Then by the monotonicity of $\frac{\pi^*\omega_\epsilon^n}{\theta_Y^n}$ in $\epsilon$, we have 
\begin{equation}\label{eqn:uniform-1+delta}
\int_Y\left(\frac{\pi^*{\omega^n_\epsilon}}{\theta_Y^n}\right)^{1+\delta_0}{\theta_Y^n}\leq C 
\end{equation} for some constant $C$ independent of $\epsilon$.   
Since the volumes $V_{\omega_\epsilon}$ are bounded between  $[\omega_Z]^n$ and $[\omega_1]^n$, up to enlarging the constant $C$, we have 
\begin{equation}\label{eqn:uniform-1+delta'}
\int_Y\left( \frac{1}{V_{\omega_\epsilon}} \cdot \frac{\pi^*{\omega^n_\epsilon}}{\theta_Y^n}\right)^{1+\delta_0}{\theta_Y^n}\leq C 
\end{equation}

For any $p>1$, and  for any smooth function $H$ on $Y$, 
we also have the following elementary inequality 
\begin{equation}\label{eqn:entropy}
    \int_Y|H|^pe^H \theta_Y^n\leq  C'+ C' \int_Y e^{(1+\delta_0)H}\theta_Y^n,
\end{equation} 
where $C'$ is a constant depending only on $(Y,\theta_Y)$,  $p$ and $\delta_0$. 
Hence the $p$-th Nash-Yau entropies of $\omega_\epsilon$ are uniformly bounded.  
\end{proof}

Note that by Lemma \ref{l:logtype}, we can write 
\[\log  \left(\frac{1}{V_{\omega_\epsilon}} \cdot \frac{ \pi^*{\omega_\epsilon^n}}{\theta_Y^n} \right)=F_{\log}+G_\epsilon, \] 
where  $G_\epsilon$ is a bounded continuous function on $Y$, 
and \[F_{\log} = a(\log |s_{E_1}|_{h_{E_1}} - \log |s_{E_2}|_{h_{E_2}} ) \] 
which  depends only on $X$ and $Y$.   
It follows that  $G_\epsilon + \log V_{\omega_\epsilon}$ is increasing in $\epsilon$.  
By \eqref{eqn:uniform-1+delta'} and by comparing with $ G_1$, we can find a positive constant  $  C''>0$, independent of $\epsilon$, such that 
\begin{equation}\label{eqn:zz}
\int_Y e^{(1+\delta_0)F_{\log}}{\theta_Y^n}\leq C'',\,\,\int_Y e^{(1+\delta_0)G_\epsilon}{\theta_Y^n}\leq C''.\end{equation}

In the following argument, 
we will approximate $\pi^*\omega_\epsilon$ by a family of smooth K\"ahler forms $\omega_{\epsilon,j}$.  
By abuse of notation,  
we will omit the subscript $\epsilon$ and set $\omega:=\omega_\epsilon$. 
We choose a smooth closed $(1,1)$-form $\omega_0 \in [\omega]$. Since $\omega$ is orbifold smooth, it has continuous local potentials. 
Hence  there exists a unique $\varphi\in PSH(X, \omega_0)\cap \mathcal{C}^0(X)$ such that
$$\omega= \omega_0+\ddbar \varphi, ~\sup_X \varphi=0. $$
We set 
$$Q:=\log \left(\frac{1}{V_\omega} \cdot \frac{(\pi^*\omega)^n}{\theta_Y^n}\right)=F_{\log} + G.  $$
Then $G$ satisfies the assumptions of  Proposition \ref{prop:appg}. 
In addition, by Proposition \ref{prop:VerifyAKclass}, we have 
$$\mathcal{N}_p(\omega) = \frac{1}{V_\omega} \int_Y |Q|^p (\pi^*\omega)^n =   \int_Y |Q|^p e^Q \cdot \theta_Y^n \leq K .$$

\begin{lemma}\label{lemma:goodapp}
We can find a sequence of smooth functions $\{Q_j\}_{j\gg 1}$ on $Y$,   
which converges to $Q$, smoothly on any compact subsets of $Y\setminus D$.  
In addition,  the  following properties hold.   
\begin{enumerate}

\item Let $\gamma$ be the function defined in \eqref{eqn:gamma}.  
There is a constant $c>0$, independent of $\epsilon$ and $j$, such that 
for all sufficiently large $j$,  we have
\[  
 e^{Q_j} \geq  c \cdot  \gamma \cdot |s_{E_2}|_{h_{E_2}}^a. 
\]

%
\item 
For any $\delta\geq 0$ small enough, there exists $K'>0$, independent of $\epsilon$ and 
$j$,   
such that for all $j>0$ sufficiently large, we have 
\begin{equation}\label{app} 
\|e^{Q_j}\|_{L^{1+\delta}(Y, \theta_Y)} \leq K'. 
\end{equation}

\medskip

\item 
%
There exists $N_\epsilon>0$ and $C_\epsilon>0$, possibly depends on $\epsilon$,  
such that  
$$ \sup_j \|\nabla^2 Q_j \|_{\theta_Y} \leq C_\epsilon |s_D|^{-2N_\epsilon}_{h_{{D}}}. $$
\end{enumerate}
\end{lemma}
\begin{proof} We approximate the two functions $F_{\log}$ and ${G}$ separately. For the approximation of $F_{\log}$, we set
$$F_j=   \frac{a}{2} \cdot \log \left(\frac{|s_{E_1}|_{h_{E_1}}^2+ j^{-1}}{|s_{E_2}|_{h_{E_2}}^2+ j^{-1}}  \right).  $$ 
Then  $\{F_j\}$  converges to $F_{\log}$ smoothly on any compact subset of $Y\setminus D$.   
We approximate the function $G$ by a sequence of smooth functions $G_j$ according to Proposition \ref{prop:appg}. 
More precisely, we may let $G_j$ be $G_{\frac{1}{j}}$ with the notation in Proposition \ref{prop:appg}.  
Let $Q_j=F_j + G_j$.    
Then we can verify that the item $(3)$ holds.

For the item (1), we first recall that $ \frac{(\pi^*\omega)^n}{\theta_Y^n} \geq \gamma$ by Proposition \ref{prop:VerifyAKclass}.   
Since $V_{\omega}$ is bounded by positive numbers independent of $\epsilon$, 
we deduce that 
\[Q= F+G\geq \log  (c'\cdot \gamma )\] 
for some constant $c'>0$ independent of $\epsilon$.   
Since $\{G_j\}$ converges to $G$ uniformly on $Y$, 
we may assume that $G_j\geq G-1$. 
Then 
\begin{eqnarray*}
Q_j- \log c'\gamma &\ge&  ( F_j - F )+ (F + G -\log c'\gamma)  -1  \\ 
& \geq &  \frac{a}{2} \log \left( \frac{|s_{E_2}|_{h_{E_2}}^2}{|s_{E_2}|_{h_{E_2}}^2+ j^{-1}}   \right) + 0 - 1. 
\end{eqnarray*} 
We note that $\frac{a}{2} \log (|s_{E_2}|_{h_{E_2}}^2+ j^{-1}) \leq \frac{a}{2} \log (|s_{E_2}|_{h_{E_2}}^2+ 1) $  is bounded from above  by some constant $\lambda$ depending only on $(E_2,h_{E_2},a)$. 
Hence we deduce that  
\[
Q_j- \log c' \gamma  \geq a \log |s_{E_2}|_{h_{E_2}}   -(1+\lambda)
\]
By setting $c=c'\cdot e^{-(1+\lambda)}$, we obtain the item $(1)$.

For the item $(2)$, we first fix some $\delta \geq 0$ small enough. 
Then the $L^{1+\delta}$-norm of $e^{F+G}$ on $(Y,\theta_Y)$ is bounded by  some constant independent of $\epsilon$, as shown in \eqref{eqn:uniform-1+delta'}.  
Since  $\{G_j\}$ converges to $G$ uniformly on $Y$, we only need to prove that  
the $L^{1+\delta}$-norms of $e^{F_j}$ are bounded by some constant, independent of $\epsilon$ and $j$.   
We have the following estimate 
\begin{equation}
\label{eqn:dominated-conv}
    e^{F_j} \le (|s_{E_1}|_{h_{E_1}}^2+1)^{\frac{a}{2}}  \cdot  |s_{E_2}|_{h_{E_2}}^{-a}.  
\end{equation}
We have seen in \eqref{eqn:zz}, that $e^{(1+\delta)F_{\log}}$ is integrable. 
Since $E_1$ and $E_2$ do not have common component, 
it  follows that $|s_{E_2}|_{h_{E_2}}^{-a(1+\delta)}$ is integrable, and so is the RHS of the inequality above.   
By the dominated convergence theorem, we deduce the following convergence, 
\[
 \|e^{F_j}\|_{L^{1+\delta}(Y, \theta_Y)} \to \|e^{F_{\log}}\|_{L^{1+\delta}(Y, \theta_Y)}. 
\]
By  \eqref{eqn:zz} again, we can deduce a uniform constant $K'$ for the item $(2)$.  
This completes the proof of the lemma.   
\end{proof}

We will now use the smooth functions $Q_j$ to construct smooth forms  approximating  $\omega$. 
Recall that  $\omega=\omega_0+\ddbar\varphi$.
Pulling back to $Y$, we have
\begin{equation}\label{eqn:ss}
 (\pi^*\omega_0+\ddbar\pi^*\varphi)^n=V_\omega\cdot e^Q\theta_Y^n.
\end{equation}
Let $\{\delta_j\}$ be a sequence of positive real numbers in $(0,1)$ converging to $0$. 
We  consider the following perturbed complex Monge-Amp\`ere equation
\begin{equation}\label{eq}
 (\pi^*\omega_0 + \delta_j \theta_Y + \ddbar \varphi_j)^n = e^{Q_j + c_j} \theta_Y^n, ~~\sup_X \varphi_j = 0,
 \end{equation}
where $c_j$ is the normalizing constant satisfying 
$$\int_Y e^{Q_j+ c_j} \theta_Y^n = \left(\pi^* [\omega_0] + \delta_{{j}} [\theta_Y] \right)^n. $$
Then the solution $\varphi_j$ exists and is smooth by Yau's theorem.  We define 
\begin{equation}\label{eqn:omegaj}
\omega_j = \pi^*\omega_0 + \delta_j \theta_Y + \ddbar \varphi_j. 
\end{equation}


\begin{lemma}\label{lemma:unisob} There exist  constants $ {A}^\circ,{K}^\circ,{p}^\circ$ and a non-negative continuous function $ {\gamma}^\circ$ on $Y$, all independent of $\epsilon$ and $j$, satisfying the following property. 
There is an integer $M_\epsilon>0$, such that $\omega_{j}\in  \mathcal{W}(Y, \theta_Y, n, p^\circ, A^\circ, K^\circ, \gamma^\circ)$ whenever $j \geq M_\epsilon$.
\end{lemma}

\begin{proof}
We observe that $[\omega_j]^n$ and $[\omega_j]\cdot [\theta_Y]^{n-1}$ are uniformly bounded. 
This gives a constant $A^\circ$.  
Moreover $[\omega_j]^n\geq [(\rho\circ \pi)^*\omega_Z]^n>0$.  
Next, we will show that $\frac{\omega_j^n}{\theta_Y^n} = e^{Q_j+c_j}$ is bounded from below by some $\gamma^\circ$.  
By the item $(1)$  of Lemma \ref{lemma:goodapp}, 
it is enough to show that, 
for  $j$ sufficiently large,  the following number 
$$ |{c_j}|=\left|\log\frac{\left(\pi^* [\omega_0] + \delta_{{j}} [\theta_Y] \right)^n}{\int_Ye^{Q_j}\theta_Y^n}\right|$$ 
is  bounded, by  a constant independent of $\epsilon$ and $j$. 
We recall that $Q_j=F_j+G_j$ such that $\{G_j\}$ converges uniformly to $G$. 
By \eqref{eqn:dominated-conv} and by using the dominated convergence theorems,  
we see that $\int_{(Y,\theta_Y)} e^{Q_j} \to \int_{(Y,\theta_Y)}e^Q$. 
It follows that the sequence $\{c_j\}$ converges to $\log  V_{\omega}$, which is bounded by constants independent of $\epsilon$.

It remains to prove that, there is some $p^\circ\geq 1$, 
such that the  $p^\circ$-th Nash-Yau entropy of $\omega_j$ 
\[
\mathcal{N}_{p^\circ}(\omega_j) =  \frac{1}{[\omega_j]^n}  \int_Y (Q_j+c_j)^{p^\circ} \cdot e^{Q_j+c_j} \cdot \theta_Y^n
\]
is  bounded by some constant $K^\circ$, for all $j$ sufficiently large.    
Let $p^\circ\geq 1$ be arbitrary.  
We have proved that, for $j$ sufficiently large,  
 $|c_j|$ and $([\omega_j]^n)^{-1}$ are bounded by constants independent of $\epsilon$ and $j$.   
Hence by using \eqref{eqn:entropy} and   \eqref{app}, 
we can deduce a uniform bound for $\mathcal{N}_{p^\circ}(\omega_j)$. 
This completes the proof of the lemma.  
\end{proof}

In order to show that the family $\{\omega_j\}$ converges to $\omega$, we first prove the  following uniform estimates for the potentials $\varphi_j$.  
\begin{lemma} 
\label{lemma:2nd} There exist $N'_\epsilon, C'_\epsilon>0$, possibly depend on $\epsilon$,  such that for all $j>0$ sufficiently large, 
$$\|\varphi_j\|_{L^\infty(Y)} \leq C'_\epsilon,  \ \Delta_{\theta_Y} \varphi_j \leq C_\epsilon'|s_D|_{h_D}^{-2N_\epsilon'}. $$

\end{lemma}
\begin{proof} Thanks to  Lemma \ref{lemma:goodapp},  
we can argue exactly as in  
\cite[Lemma 7.1]{GPSS}. 
\end{proof}

We can then deduce that  the sequence of smoothing metric $\omega_{j}$ converges locally and smoothly to $\omega$ away from the divisor $D$.  
\begin{lemma} \label{omapp} 
For any relatively compact open subset $\mathcal{K} \subset   Y\setminus D$ and for any integer $k\geq 0$, we have 
%
$$\lim_{j\rightarrow \infty} \|\varphi_j - \pi^*\varphi\|_{L^\infty(\mathcal{K})} =0, $$
$$\lim_{j\rightarrow \infty}  \|\omega_j - \omega\|_{\mathcal{C}^k(\mathcal{K})} =0.$$
%
\end{lemma}  
\begin{proof}Note that $\omega_j$ is smooth outside $D$ and the sequence $\{Q_j\}$ converges smoothly outside $D$. By the second inequality of Lemma \ref{lemma:2nd}, we have a uniform $\mathcal{C}^2$ estimates of $\varphi_j$ for any compact set $K$ inside  $Y\setminus D$. 
By Evans-Krylov theory, we can obtain local higher order estimates for $\varphi_j$,  uniformly away from $Y\setminus D$.
\end{proof}

Now we can conclude Theorem \ref{thm-orbifold-AK-property}. 

\begin{proof}[{Proof of Theorem \ref{thm-orbifold-AK-property}}] 
From Lemma \ref{lemma:goodapp} to Lemma \ref{omapp},  We have proved that for each $\omega_\epsilon$, it admits a sequence of  approximations $\{\omega_j\}$, 
belonging to  the same class $ \mathcal{W}(Y, \theta_Y, n, p^\circ, A^\circ, K^\circ, \gamma^\circ)$. 
By the same argument as in   \cite[Section 8]{GPSS},  
we can prove the statements of Theorem \ref{thm:soborbi}  for the family $\{\omega_\epsilon\}$, uniformly in $\epsilon$.   
\end{proof}

\begin{remark}\label{rmk:heatuni} For orbifold smooth K\"ahler form $\omega=\omega_\epsilon$, the existence of orbifold smooth heat kernel is known to exist \cite[Proposition 4.1]{Chiang90}. 
Following the same lines of \cite[Corollary 10.5]{GPSS}, we can verify that the orbifold heat kernel is identical with the heat kernel in Definition \ref{def:Heatkernel} for $\omega$.  
\end{remark}

We also need the following statement in the next section. 

\begin{lemma}
\label{lemma:convergence-L1} 
Let $\eta$ be a continuous function on $Y\setminus D$ such that $|\eta|$ is bounded by  $-\alpha \cdot \log |s_D|_{h_D} + \beta$, where $\alpha,\beta>0$ are constants. 
Then the following convergence holds, 
\[
\int_{(Y,\omega_j)} \eta  \to  \int_{(Y,\omega)} \eta    \mbox{ when }  j\to +\infty.  
\]
\end{lemma}

\begin{proof}
We have seen in the proof of Lemma \ref{lemma:unisob} that  the sequence $\{c_j\}$ converges. 
Hence for $j$ sufficiently large, there is a constant $\nu$, independent of $j$, 
such that 
\[
e^{Q_j+c_j} \le \nu \cdot  (|s_{E_1}|_{h_{E_1}}^2+1)^{\frac{a}{2}}  \cdot  |s_{E_2}|_{h_{E_2}}^{-a}. 
\] 
We notice that the product of $-\alpha \cdot \log |s_D| + \beta$ with the RHS above is integrable on $(Y,\theta_Y)$.  
Hence we can conclude by using the dominated convergence theorem.  
\end{proof}

\section{Uniform $\mathcal{C}^0$-estimates on  Hermitian-Einstein metrics}  
\label{section:HE}

We fix the following notation for this section. 
Let $(Z,\omega_Z)$ be a compact K\"ahler variety of dimension $n$,  
which has quotient singularities in codimension $2$,  
and let $\mathcal{F}$ be a reflexive coherent sheaf on $Z$.  
We assume that $\mathcal{F}$ is $\omega_Z$-stable. 
Let $\rho \colon  X\to Z$ be an orbifold modification so that there is an orbifold structure $\mathfrak{X}$ over $X$.  We denote $\mathcal{E} = (\rho^*\mathcal{F})^{**}$. 
We may assume that there is an orbifold vector bundle $\mathcal{E}_{\orb}$ on $\mathfrak{X}$, which descend to $\mathcal{E}$, away from the $\rho$-exceptional locus and the branched locus of $\mathfrak{X}$,  see \cite[Section 9]{Ou2024}.  
We emphasize that, by construction, the indeterminacy locus of $\rho^{-1}$ has codimension at least $3$ in $Z$, and the codimension $1$ part of the branched locus of $\mathfrak{X}$ is $\rho$-exceptional. 
In addition, we can assume that there is some $\rho$-exceptional $\rho$-ample divisor (see \cite[Remark 8.2]{Ou2024}).      
Let $\omega_{\orb}$ be a K\"ahler current on $X$ which corresponds to  an orbifold K\"ahler form, and let $\omega_\epsilon = \rho^*\omega_Z +  \epsilon\omega_{\orb}$ for all $ 0< \epsilon\leq 1$.   
Without loss of the generality, we  assume that  $\omega_\orb \geq \rho^*\omega_Z$.  
Then the orbifold  vector bundle $\mathcal{E}_{\orb}$ is stable with respect to $\omega_\epsilon$ for all $\epsilon>0$ small enough by \cite[Claim 9.5]{Ou2024}. 
We fix an orbifold smooth Hermitian metric $h$ on $\mathcal{E}_\orb$.  
By abuse of notation, we also denote by $h$ the induced metrics on $\mathcal{E}$, 
which is well-defined at least on some dense Zariski open subset of $X$.  
Let $(\mathcal{L}_\orb, h_{\mathcal{L}})$ be the determinant line bundle of $(\mathcal{E}_\orb,h)$, and $\theta_\mathcal{L}$ be the Chern curvature of $h_{\mathcal{L}}$.  
Then $\theta_{\mathcal{L}}$ can also be viewed as a current on $X$ which is orbifold smooth.


Let $\pi\colon Y\to X$ be a log resolution of the closed analytic subset $\Sigma\subseteq X$, 
where $\Sigma$ is the union of the 
branched locus of the orbifold structure $\mathfrak{X}$ and  the $\rho$-exceptional locus.   In particular, the $(\rho\circ \pi)$-exceptional locus is a snc divisor.   
We note that the $(\rho\circ \pi)$-exceptional locus contains the $\pi$-preimage of the branched locus of the orbifold structure $\mathfrak{X}$, by the construction of $\rho$.    
We choose  an effective divisor $D$ on $Y$ whose support is equal this exceptional divisor, 
so that   $[(\rho\circ \pi)^*\omega_{Z}] - \delta[D]$ is a K\"ahler class on $Y$ for all $\delta>0$ small enough.  
Let $s_D\in H^0(Y,\mathcal{O}_Y(D))$ be a section defining $D$, 
and let  $h_D$ be a smooth Hermitian metric on the line bundle $\mathcal{O}_Y(D)$, 
so that 
\[(\rho\circ \pi)^*\omega_{Z} -\delta \ddbar\log |s_D|_{h_D}\]  
is a K\"ahler form on $Y$ for all $\delta>0$ small enough.

We note that,  throughout this section,  
all possibly singular metrics, functions or currents are indeed smooth objects defined on the largest Zariski open sets where $Y,X,Z$ are isomorphic.   
Therefore, by abuse of notation, we may use the same letter for such objects, 
which are eventually the same  on these isomorphic open sets,  
without specifying the compactifications $Y,X,Z$.

We introduce some quantities related to Hermitian-Einstein metrics for this section.  
By \cite[Theorem 1]{Faulk2022}, the orbifold vector bundle $\mathcal{E}_{\orb}$ admits an orbifold Hermitian-Einstein metric  $h_{\epsilon,HE}$ with respect to $\omega_\epsilon$.  
We can interpret these metrics as follows,  
\begin{equation*}
h_{\epsilon,HE} =: h\cdot e^{- \frac{1}{r}\rho_\epsilon}\exp(s_\epsilon),\,\,\,H_\epsilon:= e^{- \frac{1}{r}\rho_\epsilon}\exp(s_\epsilon) ,\,\,\,
 S_\epsilon:= \exp(s_\epsilon),
\end{equation*}
where $s_\epsilon$ is an $h$-self-adjoint endomorphism of $\mathcal{E}$ such that $\Tr s_\epsilon= 0$.  
We  have the equality 
\[
\log \Tr H_\epsilon= \log \Tr S_\epsilon- \rho_\epsilon.
\]
The Einstein condition implies that $\rho_\epsilon$ satisfies the following equation
\begin{equation}\label{eqn:a}
\Lambda_{\epsilon} \theta_\mathcal{L}+ \Delta_{\epsilon}'\rho_{\epsilon}= 
\frac{1}{Vol(X, \omega_{\epsilon})}\int_{X}c_1(\mathcal{E},h)\wedge (\omega_{\epsilon} )^{n-1},  
\end{equation}
where $Vol$ is the volume, 
 $\Delta_\epsilon$ is the Laplace-Beltrami operator for $\omega_\epsilon$ and $\Delta_\epsilon'= \frac{1}{2}\Delta_\epsilon$.  
Up to adding a constant, we may assume that $\rho_\epsilon$ is the unique  solution so that 
\[\displaystyle \int_{(X,\omega_\epsilon)}\rho_{\epsilon} = 0,\] 
see \cite[Theorem 2.6]{Chiang90}.  
Since $\Lambda_\epsilon\theta_L$ is orbifold smooth, so is the solution  $\rho_\epsilon$.

The main objective of this section is to  show that, there is  a sequence of  $h_{\epsilon,HE}$, which converges to a Hermitian-Einstein metric with respect to $\rho^*\omega_Z$ as $\epsilon\rightarrow 0$.   
The key   is to prove certain uniform $\mathcal{C}^0$ estimates on the endomorphisms  $H_{\epsilon}$, see Proposition \ref{prop:C0} for the precise statement. 

We remark that the case when $h_{\epsilon,HE}$ are smooth with respect to a degenerate family of K\"ahler forms is addressed in \cite{CGNPPW}. 
In our case, the new difficulty is the lack of uniform geometric estimates of the family   $\{\omega_\epsilon\}$, 
which are proved in Section \ref{section:uniform-kahler}. 
Essentially, 
this is the only  different part for the convergence of $h_{\epsilon,HE}$,   comparing with \cite{CGNPPW}.  We also remark that it may be possible to take subsequential limit by using compactness result on  Hermitian-Yang-Mills connections in  \cite{Tian00}.

\subsection{Uniform mean value type inequalities for Hermitian-Einstein metrics}
The purpose of this subsection is to prove two  uniform mean value type inequalities for the Hermitian-Einstein metrics $h_{\epsilon,HE}$.  
We adapt the method  of \cite[Section 2.2]{CGNPPW}


\begin{lemma}\label{lemma:boundrho}
There exists   positive constants $C,C'>0$, independent of $\epsilon$, such that the following inequalities hold for all $\epsilon>0$ small enough.
\begin{equation*}  C' \log |s_D|_{h_D}^{2}+ |\rho_\epsilon| \leq  C\left(1 + \int_{(X,\omega_\epsilon)}|\rho_\epsilon| \right).  
\end{equation*}  
\end{lemma}

\begin{proof}  
We will first establish the following inequality 
\begin{equation}\label{eqn:rho+} 
C' \log |s_D|^{2}+ \rho_\epsilon \leq  C\left(1 + \int_{(X,\omega_\epsilon)}|\rho_\epsilon| \right). 
\end{equation}  
By assumption at the beginning of the section, there is some effective $\rho$-exceptional Cartier divisor $D_X\subseteq X$ so such that $\rho^*\omega_Z-\delta[D_X]$ is a K\"ahler class for some $\delta>0$ small enough. 
It is then an orbifold K\"ahler class as well. 
Let $\sigma \in H^0(X,\mathcal{O}_X(D_X))$ be a section defining $D_X$.
Then, by $\partial\bar{\partial}$-lemma for compact K\"ahler orbifolds (see for example \cite[Section 7]{Baily56}),  there is some orbifold smooth Hermitian metric $\gamma$ on $\mathcal{O}_X(D_X)$, such that $\rho^*\omega_Z+ \delta \ddbar \log|\sigma|_\gamma^{2}$ is an orbifold 
K\"ahler form.  
In particular, $\gamma$ is continuous.  
Since $\theta_\mathcal{L}$ is orbifold smooth,  there is some constant $C_1>0$ such that 
\begin{align*}
\theta_\mathcal{L}\leq C_1(\rho^*\omega_Z+  \delta \cdot \ddbar \log|\sigma|_\gamma^{2}).  
\end{align*}  
Since $\rho^*\omega_Z\leq \omega_\epsilon$,   we deduce that 
\begin{align*}
\theta_\mathcal{L}\leq C_1(\omega_\epsilon+ \delta \cdot \ddbar \log|\sigma|_\gamma^2  ).  
\end{align*}  
We set $\eta= C_1\cdot \delta \cdot  \log|\sigma|_\gamma^2$.  
Then we have 
\begin{equation*}\label{eqn:mean2}
\Lambda_\epsilon \theta_L\leq C_1+ \Delta_\epsilon'\eta.  
\end{equation*}
Combining with \eqref{eqn:a}, we obtain that 
\begin{equation}\label{eqn:mean2'} 
  \Delta_{\epsilon}'(\rho_{\epsilon} + \eta)\geq -C_1 + 
\frac{1}{Vol(X, \omega_{\epsilon})}\int_{X}c_1(\mathcal{E},h)\wedge ( \omega_{\epsilon})^{n-1}.
\end{equation}

Recall that, for a fixed $\epsilon$,  
there is a  sequence of smooth K\"ahler form  $\omega_{j}$ on $Y$ which converges  to  $\omega_\epsilon$ outside $D$, 
see   \eqref{eqn:omegaj}.  
We note that  $\pi^*\eta$  has at most log poles along $D$ since the support of $\pi^*D_X$ is contained in the one of $D$.  
Hence, by \eqref{eqn:uniform-1+delta} and H\"older inequality, 
we see that the integral 
\[
\int_{(X,\omega_\epsilon)} \eta 
\]
is    bounded by constants independent  of $\epsilon$.   
We also recall that   $\int_{(X,\omega_\epsilon)} \rho_\epsilon  =0$.  
Hence, by Lemma \ref{lemma:convergence-L1}, there is a constant $I$ independent of $\epsilon$ and $j$, such that 
\[
|\int_{(Y, \omega_j)} (\eta +\rho_\epsilon)   |\leq I
\]
for all $j$ sufficiently large.  
We also observe from \eqref{eqn:omegaj}, that the volume of $(Y,\omega_j)$ is bounded from below by   $[(\rho\circ\pi)^*\omega_Z]^n >0$ and from above by $( [(\rho\circ\pi)^*\omega_Z] + \rho^*[\omega_\orb] + [\theta_Y])^n$.  
Hence there is some constant $B>1$, independent of $\epsilon$ and $j$, such that the volume of $(Y,\omega_j)$ is contained in $[B^{-1},B]$.

Since $\rho_\epsilon(x)+\eta$ goes to $-\infty$  when $y\in Y $ approaches to $D$,  
the set $$K_\epsilon:=\{y\in Y \ |\ \rho_\epsilon+\eta \geq -IB -1\}$$ is a compact subset of  $Y\setminus D $.    
Let $\Delta_j$ be the Laplace-Beltrami operator with respect to $\omega_j$.  
Then on $K_\epsilon$, we have the following smooth convergence 
\begin{equation*}\label{eqn:laplace-converge}
    \Delta_j (\rho_\epsilon+\eta )\rightarrow \Delta_\epsilon (\rho_\epsilon+\eta).
\end{equation*} 
From  \eqref{eqn:mean2'}, the RHS above is bounded from below by some constant independent of $\epsilon$. 
Hence, there is some constant $a$ independent of $\epsilon$ and $j$, such that for all   $j$ sufficiently large, 
we have the following inequality on $K_\epsilon$, 
$$\Delta_j (\rho_\epsilon+\eta)\geq a.$$
By Lemma \ref{lemma:unisob}, for $j$ sufficiently large, we have $\omega_{j}\in \mathcal{W}(Y, \theta_Y, n, p, A, K, \gamma)  $ for some $K,A,p,\gamma$ independent of $\epsilon$ and $j$.    
Hence, by Lemma \ref{lemma:meanvalue}, there is a constant $C_2$ independent of $\epsilon$ and $j$ such that 
$$\rho_\epsilon+\eta\leq C_2 \left(1+\int_{(Y,\omega_j)}(|\rho_\epsilon|+|\eta|) \right)$$ 
for all $j$ sufficiently large.  
We recall that $\omega_j^n = e^{Q_j+c_j}\cdot \theta_Y^n$ and the sequence of numbers $\{c_j\}$ converges to $\log Vol(X,\omega_\epsilon)$. 
Thus, by using \eqref{app} of Lemma \ref{lemma:goodapp} and H\"older inequality, 
we can obtain a uniform upper bound on  $\| \eta \|_{L^1(Y,\omega_j)}$, for all $j$ sufficiently large.   
It follows that  
$$\rho_\epsilon+\eta \leq C_3  \left(1+\int_{(Y,\omega_j)}|\rho_\epsilon|   \right), $$ 
for some constant $C_3>0$ independent of $\epsilon$ and $j$. 

Since $\rho_\epsilon$ is bounded, by Lemma \ref{lemma:convergence-L1}, 
the integral in the RHS above converges to $\int_{(X,\omega_\epsilon)}|\rho_\epsilon| $ when $j$ tends to $+\infty$.  
Therefore, we obtain that 
\[
\rho_\epsilon + \eta \le C_3 \left(   1 + \int_{(X,\omega_\epsilon)} |\rho_\epsilon|  \right). 
\]

It remains to compare $\eta$ with $\log |s_D|_{h_D}^2$.  
Since the support of $\pi^*D_X$ is contained in the one of $D$, 
and since $\gamma$ is continuous, 
we see that 
\[
A\cdot \log |s_D|_{h_D}^2 \leq   \pi^*\log |\sigma|_\gamma^2 + A'
\]
for some constants $A' \gg A >0$ sufficiently large.  
Hence there is a constant $B', C'>0$ such that $ C'\log |s_D|_{h_D}^2 \le \eta +B'$. 
This completes the proof of  \eqref{eqn:rho+}.  

By replacing $\theta_\mathcal{L}$ and $\rho_\epsilon$ by $-\theta_\mathcal{L}$ and $-\rho_\epsilon$ respectively in the previous reasoning, we see that \eqref{eqn:rho+} still holds if we replace $\rho_\epsilon$ by $-\rho_\epsilon$, up to adjusting the constants $C,C'$.  
This completes the proof of the lemma. 
\end{proof}
\begin{remark} We remark that, we do not use the Heat kernel estimates for orbifold metrics directly when deriving the mean value inequality. Since the function $\log|s_D|^2_{h_D}+\rho_\epsilon$ has some log poles, it is not very clear if we can use its Laplacian and the heat kernel to represent this function.
\end{remark}

We also have the following estimates. 

\begin{lemma}\label{lemma:boundH}
There are  constants $C,C'>0$ such that the inequality 
\begin{equation*} C'\log |s_D|^2_{h_D} + 
\log \Tr H_\epsilon\leq C\big(1+ \int_{(X,\omega_\epsilon)}\log \Tr H_\epsilon  \big),
\end{equation*}
holds for every $\epsilon\in (0, 1]$.
\end{lemma}

\begin{proof} By \cite[Formula (1.9.2)]{Siu87} or \cite[Lemma 3.1]{Simpson1988}, 
we have 
\begin{equation*}
\Delta_{\epsilon}'\left(\log \Tr H_\epsilon\right)\geq - \| \Lambda_\epsilon \Theta_h \|_h - \|  \Lambda_\epsilon \Theta_{h_{\epsilon,HE}}\|_{h_{\epsilon,HE}}, 
\end{equation*} 
where  $\Theta$ stands for the Chern curvature tensor.  
Then Einstein condition implies that $\|\Lambda_\epsilon \Theta_{h_{\epsilon,HE}}\|_{h_{\epsilon,HE}}$ is  bounded by a constant  independent of $\epsilon$.  
The remainder of the proof is similar to the one of Lemma \ref{lemma:boundrho}. 
We will just mention several main step here. 
As   in  the proof of Lemma \ref{lemma:boundrho}, there is some function $\psi:=\delta \cdot \log |\sigma|_\gamma^2$,  which has log poles along $D_X$, such that $\rho^*\omega_Z + \ddbar \psi$ is an orbifold K\"ahler form.  
Since $\Theta_h$ is orbifold smooth,  there is some  constant $A>0$, independent of $\epsilon$,  such that   
\begin{equation*}
-  A(\omega_\epsilon+     \ddbar  \psi) \cdot \mathrm{Id} \leq  \Theta_h \leq A(\omega_\epsilon+   \ddbar \psi ) \cdot \mathrm{Id},   
\end{equation*} 
where the inequalities $\leq$ are considered in the sense of Nakano positivity. 
It follows that, 
\begin{equation*}
-  A(1 +  \Delta_\epsilon' \psi ) \cdot \mathrm{Id} \leq \Lambda_\epsilon \Theta_h \leq A(1 +  \Delta_\epsilon'   \psi ) \cdot \mathrm{Id}, 
\end{equation*}
where the inequalities $\leq$ are considered for $h$-self-adjoint  endomorphisms.  
Since $ \Lambda_\epsilon \Theta_h$ is self-adjoint with respect to $h$,  
we deduce that, if   $\eta:=  \rk (\mathcal{E})^{\frac{1}{2}} \cdot  A\cdot \psi$, then   
\[
\|\Lambda_\epsilon \Theta_h\|_h \leq  \rk (\mathcal{E})^{\frac{1}{2}} \cdot A +\Delta_\epsilon' \eta. 
\]
Hence we get 
\begin{equation*}
\Delta_{\epsilon}\left(\eta + \log \Tr H_\epsilon \right)\geq A'
\end{equation*} 
for some constant $A'$. 
Arguing as in Lemma \ref{lemma:boundrho}, 
where we consider $\log \Tr H_\epsilon $ in the place of $\rho_\epsilon$, 
we deduce that 
\[\eta + \log \Tr H_\epsilon \leq B(1+ \int_{(X,\omega_X)} \log \Tr H_\epsilon ) \] for some constant $B>0$. 
By comparing $\eta$ with $\log |s_D|_{h_D}^2$, 
we can obtain  the inequality of the lemma.  
\end{proof}


\subsection{$\mathcal{C}^0$ estimate of $h_{\epsilon,HE}$ with barrier}
The main purpose of this subsection is to prove the following $\mathcal{C}^0$ estimate of $H_\epsilon$ and $\rho_\epsilon$, from which Theorem \ref{thm:BG-inequality-intro} follows directly. 
Here additional care should be paid to orbifold singularities, which cause no serious trouble after the preparations of previous discussions.
Recall that  $H_\epsilon \in End_h(\mathcal{E})$ defines a Hermitian-Einstein metric  with respect to $\omega_\epsilon$ by  $h_{\epsilon,HE}=hH_\epsilon$.    
We set $\eta_\epsilon:= \log H_\epsilon$ and recall that 
\[ \rho_\epsilon = -\Tr \eta_\epsilon,  \ \ \ \ \  
\eta_\epsilon =  -\frac{1}{\mathrm{rk}(\mathcal{E})}\rho_\epsilon \otimes \ID+  s_\epsilon,  \ \ \ \ \  S_\epsilon = \exp(s_\epsilon).   
\]

\begin{prop}\label{prop:C0} There exists  constant $C,C'> 0$, independent of $\epsilon, $ such that the following inequalities  hold
\begin{equation}\label{est1}
\Tr  H_\epsilon \leq C -C' \log |s_D|_{h_D}^2, \qquad
| \rho_\epsilon| \leq  C -C' \log |s_D|_{h_D}^2. 
\end{equation}
\end{prop}
\medskip

\begin{proof}
The key is to prove Lemma \ref{lemma:integral-bound} below. 
Admitting this lemma for the time being.   
Since $S_\epsilon=\exp(s_\epsilon)$ and $\Tr s_\epsilon=0$, we see that $\Tr S_\epsilon \ge 1$. Hence the  second  inequality follows from Lemma \ref{lemma:boundrho}.  
Since $\log \Tr H_\epsilon = \log \Tr S_{\epsilon} - \rho_\epsilon$, 
we can obtain the first inequality by combing  \eqref{eqn:case1-assumption} with Lemma \ref{lemma:boundH}.    
\end{proof}

\begin{lemma}
\label{lemma:integral-bound}  
There exists a constant $C> 0$ independent of $\epsilon$,  such that 
\begin{equation}\label{eqn:case1-assumption}
    \int_{(X,\omega_\epsilon)}\big(|\rho_\epsilon|+ \log {\Tr S_\epsilon}\big)   \leq C
\end{equation}
for all positive $\epsilon$. 
\end{lemma}

 
\begin{proof}
The idea is to adapt the methods of \cite{UhlenbeckYau1986} and \cite{Simpson1988},  by using  blow-up analysis. 
Assume by contradiction that the lemma does not hold.   
Then there exist sequences $(\delta_i)_{i\geq 1}$ and  $(\epsilon_i)_{i\geq 1}$ of numbers in $(0,1)$ 
converging towards zero such that 
\begin{equation} 
\label{contraseq} 
\int_{(X,\omega_i)}\big(|\delta_i\rho_i|+ \delta_i \log {\Tr S_i}\big)  = 1
\end{equation}
Here, we denote  $\omega_i$,  $\rho_i$, $s_i$, $\eta_i$ and $S_i$ for  $\omega_{\epsilon_i}$ $  \rho_{\epsilon_i}$, $s_{\epsilon_i}$, $\eta_{\epsilon_i}$ and $ S_{\epsilon_i}$  respectively.   
Let 
\[u_i:= \delta_i\eta_i = -\frac{\delta_i}{\mathrm{rk}(\mathcal{E})}\rho_i\otimes \ID_E+ \delta_is_i.\] 
We will show that, up to passing to a subsequence, $u_i$ converges to some limit $u_\infty$, which produces a destabilizing subsheaf of $\mathcal{F}$. 
This will contradict  the stability assumption on $\mathcal{F}$.

In the following reasoning, the capital letters $C$ and $C'$ denote positive real numbers, which may change from line to line. Nevertheless, they are always independent of $i$.   
Since $\det(S_i)=1$, we have $\Tr(S_i)\geq \mathrm{rk}(\mathcal{E})$. 
In particular, $\log\Tr S_i \geq 0$. 
By \eqref{contraseq},  we have 
\begin{equation*}
\int_{(X,\omega_i)}|\delta_i\rho_i| \leq 1, \, \int_{(X,\omega_i)}\delta_i \log {\Tr S_i}   \leq 1. \end{equation*}
Then by Lemma \ref{lemma:boundrho}, we deduce that 
\begin{equation}\label{ss}
|\delta_i\rho_i|\leq C-C'\cdot \delta_i\log|s_D|_{h_D}^2, 
\end{equation}
for some constants $C$ and $C'$ independent of $i$.
Recall that \begin{equation}\label{qq}
\log \Tr  H_i = \log \Tr  S_i - \rho_i,
\end{equation} 
so by \eqref{contraseq}, we get  
\[\int_{(X,\omega_i)}\delta_i |\log {\Tr H_i}|   \leq 1.\]
Then by Lemma \ref{lemma:boundH}, we deduce that
\begin{equation}\label{tw1}
 |\delta_i\log  \Tr H_i |\leq C -C'\cdot \delta_i\log|s_D|_{h_D}^2.  
\end{equation}
Combining  \eqref{ss}, \eqref{qq} and \eqref{tw1}, we obtain that   
\[
\delta_i\log\Tr S_i \leq C- C'\cdot \delta_i\log|s_D|_{h_D}^2.
\]
For a  point $x\in X$, if the largest eigenvalue of $ s_i (x)$ is $\lambda_{i,max}$, then $\lambda_{i,max} \ge 0$ for $\Tr s_i =0$. 
Moreover, since $S_i=\exp(s_i)$, we see that  
\[
\delta_i \cdot \lambda_{i,max} \leq \delta_i \cdot \log \Tr S_i. 
\]
By using $\Tr s_i=0$ again, we have $\lambda_{i,max}^2\ge \frac{1}{\rk (\mathcal{E})^3} \|s_i\|_h^2$. 
Since 
\[
\|u_{i}\|_h\leq  \rk(\mathcal{E})^{-\frac{1}{2}} |\delta_i\rho_i| + \delta_i \| s_i \|_h,
\] 
it follows that  
\begin{equation}\label{ubarrier}
\|u_{i}\|_h\leq  C -C'\cdot \delta_i\log|s_D|_{h_D}^2.
\end{equation} 

\medskip

The important step towards the contradiction we are looking for is the 
following result.
\begin{claim}\label{claim:limit} There exist a subsequence of $(u_i)_{i\geq 1}$ converging weakly to a limit
$u_\infty$ on compact subsets of $X\setminus \pi(D)$ such that the following hold. 
Let $\omega_\infty:= \rho^*\omega_Z$. 
\begin{enumerate}

\item  The endomorphism $u_\infty$ is non zero and it belongs to the space $L_{1}^2(X,\omega_\infty)$.  
In other words, both $u_\infty$ and $\dbar u_\infty$   are in $L^2(X,\omega_\infty)$.  
 
\item  Let $\Psi:\mathbb R\times \mathbb R\to \mathbb R_{> 0}$ be a smooth, positive function such that $\displaystyle \Psi(a, b)< \frac{1}{b-a}$ holds for any
$a< b$. Then we have 
$$0 \geq  
\int_{(X,\omega_\infty)}\left\langle \Psi(u_\infty)(\partial u_\infty), \partial u_\infty\right\rangle  + \int_{(X,\omega_\infty)} \Tr \big(u_\infty \Lambda_{\infty}\Theta_h\big) $$
where $\Lambda_\infty$ 
is the contraction with 
$\omega_\infty$, 
and $\langle \cdot , \cdot \rangle_\infty$ is the inner product induced by $\omega_\infty$. 
\end{enumerate}
\end{claim}

Admitting the claim for the time being, we will argue as in \cite[Section 5]{Simpson1988}.   
We remark that, in \cite{Simpson1988}, the functions $\Psi$ are  assumed to be bounded by $\frac{1}{a-b}$ when $a> b$, which are slightly different from our setting.  
However, the same argument remains valid.  
More precisely, in our situation, we  replace $\Phi(\lambda_1, \lambda_2)$ of \cite[Lemma 5.5 and Lemma 5.6]{Simpson1988} by $\Phi(\lambda_2,\lambda_1)$. 
Afterwards, we replace $\Phi_\gamma(y_1,y_2)$ of \cite[Lemma 5.7]{Simpson1988} by 
$\Phi_\gamma(y_1,y_2) = (1-p_\gamma(y_1))\cdot dp(y_1,y_2)$.  
Now, the item (2) of the claim implies that the eigenvalues of $u_{\infty}$ are constant almost everywhere on $X$, by the arguments of \cite[Lemma 5.4 and Lemma 5.5]{Simpson1988}. 
They are not all equal, since by the second inequality of  \eqref{est1}, 
we have $\Tr u_\infty =0$. 
By the same argument as \cite[Lemma 5.7]{Simpson1988},  
we can construct a saturated destabilizing subsheaf of $\mathcal{F}|_{Z^\circ}$ by using $u_\infty$, 
where $Z^\circ\subseteq Z$ is a smooth open subset whose complement has codimension at least 2, such that $\mathcal{F}|_{Z^\circ}$ is locally free.   
Such a destabilizing subsheaf extends to a coherent subsheaf of $\mathcal{F}$ by Lemma \ref{lemma:extension-subsheaf}.  
This contradicts the stability assumption on $\mathcal{F}$, 
and finishes the proof of  Proposition \ref{prop:C0}.
\end{proof}

It remains to prove the previous claim. 

\begin{proof}[Proof of Claim \ref{claim:limit}]  
The proof is quite long, and we will divide it into several steps. \\

\textit{Step 1.} We will first prove some uniform integrability. 
Since $2\omega_\orb \geq \omega_i$ by  our choice of $\omega_\orb$,  
from \eqref{ubarrier}, 
we deduce that 
\begin{equation}\label{eqn:uniform-int}
\|u_i\|_h \omega_i^n  \leq B_1  \cdot \omega_\orb^n    
\end{equation} 
for  all $i$, 
where $B_1$ is a positive function which only has log poles along $D$, and is smooth elsewhere.   
Next, since $h$ is orbifold smooth, there is some constant $A>0$ such that   
$$-A \cdot \omega_{\orb}\cdot \ID\leq \Theta_h\leq A\cdot  \omega_{\orb}\cdot \ID.$$
Hence there is some constant $A'$ such that 
$$-A'\cdot   {\omega_{\orb}\wedge\omega_i^{n-1}} \cdot \ID  
\leq\Lambda_i\Theta_h \cdot \omega_i^n  
\leq  A'\cdot {\omega_{\orb}\wedge\omega_i^{n-1}}  \cdot \ID.$$
Since $2\omega_{\orb} \geq \omega_i$, we get 
\[
\|\Lambda_i\Theta_h \|_h \cdot \omega_i^n\leq A'' \cdot  \omega_{\orb}^n 
\]
for some constant $A''>0$. 
Together with \eqref{ubarrier}, this implies   that 
\begin{equation}
\label{eqn:uni-int-2}    
\|u_i\|_h \cdot \|\Lambda_i\Theta_h\|_h \cdot  \omega_i^n  \leq  B_2 \cdot  \omega_\orb^n
\end{equation}
for all $i$, 
where $B_2$ is a positive function which only has log poles along $D$, and is smooth elsewhere.    


Let $ \Phi(x, y) = \frac{\exp(x-y)- 1}{x-y}$. 
By applying Lemma \ref{lemma:etabound} to $\eta_i=\delta_i^{-1} u_i$, we have  
\begin{equation}\label{hh}
\frac{1}{\delta_i}\int_{(X,\omega_i)}\left\langle \Phi(u_i/\delta_i)(\partial u_i), \partial u_i\right\rangle_i  
+ \int_{(X,\omega_i)} \Tr \big(u_i\Lambda_{i}\Theta_h\big) = 0,   
\end{equation}   
where $\langle \cdot , \cdot \rangle_i$ is the inner product induced by $h$ and $\omega_i$. 
Using  \eqref{eqn:uni-int-2}, we deduce that the following integrals 
\[
\frac{1}{\delta_i}\int_{(X,\omega_i)}\left\langle \Phi(u_i/\delta_i)(\partial u_i), \partial u_i\right\rangle_i  
\]
are uniformly bounded. \\

\textit{Step 2.}  Assume that $K$ is a relatively compact open subset of $X\setminus \pi(D)$. 
In this step, we will  prove a uniform estimate of the $L^2$-norms of $\partial u_i$ on $K$, 
which will imply the convergence of $u_i$,  up to passing to a subsequence.  
By \eqref{ubarrier},  the eigenvalues of $u_i$ on $K$ are contained some segment $[\alpha,\beta]$ independent of $i$. Hence, by Lemma \ref{lemma:Phi-monotone},  we have
\begin{equation*}
\int_{(K,\omega_i)}\|\partial u_i\|_i^2 
\leq C_K\int_{(K,\omega_i)}\left\langle \Phi(u_i)(\partial u_i), \partial u_i\right\rangle_i 
\end{equation*}
for some constant  $C_K$ depending only on $\alpha,\beta$.     
Here the norm $\|\partial u_i\|$ is induced by $h$ and $\omega_i$.  
Since $\delta_i\le 1$, by Lemma \ref{lemma:Phi-monotone}, we deduce that 
\begin{equation*}
\int_{(K,\omega_i)}\|\partial u_i\|_i^2 
\le 
C_K \cdot \frac{1}{\delta_i}\int_{(X,\omega_i)}\left\langle \Phi(u_i/\delta_i)(\partial u_i), \partial u_i\right\rangle_i.     
\end{equation*} 
From Step 1, we know that the RHS above is bounded from above, uniformly in $i$. 
Hence $\int_{(K,\omega_i)}\|\partial u_i\|_i^2 $ is uniformly bounded.  
Since $\omega_\infty$ is a K\"ahler form in a neighborhood of  $\overline{K}$, this implies that  $\int_{(K,\omega_\infty)}\|\partial u_i\|^2 $ is uniformly bounded, where the norm $\|\partial u_i\|$ is induced by $h$ and $\omega_\infty$.  


We have proved in \eqref{ubarrier} that the functions $\|u_i\|_h$ are uniformly bounded on $K$. 
By a standard diagonal procedure, up to passing to a subsequence,  we can  assume that the sequence $\{u_i\}$  converges weakly to an endomorphism 
$u_\infty$,  inside $L_1^{2}(X,\omega_\infty)$, on any relatively compact open subsets  of $X\setminus \pi(D)$.  Moreover, by  Rellich–Kondrachov theorem, we have the strong $L^2$ convergence
\[
\Vert u_i-u_{\infty}\Vert_{L^2(K,\omega_\infty)}\to 0  \mbox{ when } i\to +\infty.
\]
for any relatively compact open subset $K\subset X\setminus D$.

By dominated convergence theorem, \eqref{eqn:uni-int-2} implies the following convergence, 
\begin{equation}\label{eqn:converge-uLambda} 
\int_{(X,\omega_i)} \Tr \big(u_i\Lambda_{i}\Theta_h \big) 
\to 
\int_{(X,\omega_\infty)} \Tr \big(u_\infty\Lambda_{\infty}\Theta_h\big)  
\mbox{ when }  i \to +\infty. 
\end{equation} 
\\

\textit{Step 3.} In this step, we will  show that the limit $u_{\infty}$ is not identically zero. 
By the assumption of \eqref{contraseq}, we have 
\begin{equation*}
\int_{(X,\omega_i)}\delta_i |\rho_i| + \int_{(X,\omega_i)}\delta_i \log {\Tr S_i}  =  1.  
\end{equation*} 
Since $S_i=\exp(s_i)$, we have 
\[
\delta_i\log{\Tr S_i }\leq \delta_i \|s_i\|_h + \delta_i \log \rk(\mathcal{E}). 
\]
From the definition of $u_i$, we  we get
\begin{equation*}
\delta_i\log{\Tr S_i}\leq   \|u_i\|_h+ \rk(\mathcal{E})^{-\frac{1}{2}}\cdot \delta_i |\rho_i| + \delta_i \log \rk(\mathcal{E}).  
\end{equation*}
Combine with the first equality in Step 3, we deduce that 
\begin{equation}\label{10}
(1+\rk(\mathcal{E})^{-\frac{1}{2}})\int_{(X,\omega_i)}\delta_i |\rho_i| + 
\int_{(X,\omega_i)}\|u_i\|_h   \geq 1-\delta_i \cdot  Vol(X,\omega_i) \cdot \log \mathrm{rk}(\mathcal{E}),  
\end{equation} 
where $Vol$ is the volume. 
Since $\delta_i\rho_i = -\Tr u_i$, 
we deduce  the following convergence, almost everywhere on $X$,  
\[
\delta_i \rho_i \to \rho_\infty := -\Tr u_\infty \mbox{ when } i\to +\infty. 
\]
We tend $i$ to the infinity in  \eqref{10}. 
Recall that $\omega_i \le 2\omega_\orb$.   
Thanks to   \eqref{ss} and   \eqref{eqn:uniform-int}, by dominated convergence theorem, 
we can interchange limit symbol and integral symbol for the LHS of   \eqref{10}. 
It follows that $0 \geq 1$. 
This is a contradiction. 
\\

\textit{Step 4.} We will prove the item $(2)$ in this step. 
Fix a function $\Psi$ as in the statement of the claim.  
We will show that for each relatively compact open subset $K\subset X\setminus \pi(D)$, 
the following inequality hold for all $i$ sufficiently large. 
\begin{equation}\label{fffffff}
\int_{(K,\omega_i)}\left\langle \Psi(u_i)(\partial u_i), \partial u_i\right\rangle_i  + \int_{(X,\omega_i)} \Tr \big(u_i\Lambda_{i}\Theta_h\big) \leq 0
\end{equation}  
By Lemma \ref{lemma:Phi-monotone}, we note that  
${\delta_i}^{-1}\Phi(\delta_i^{-1}a,\delta_i^{-1}b)$ tends to $\frac{1}{b-a}$ if $b> a$, 
and to $+\infty$ if $ b\le a$.  
As we have seen before that, by   \eqref{ubarrier},   
the eigenvalues of  $u_i$ on $K$ are contained in an bounded segment $[\alpha,\beta]$.  
Hence, for some $i$ sufficiently large, we have 
\[
{\delta_i}^{-1}\Phi(\delta_i^{-1}a,\delta_i^{-1}b) \geq \Psi(a,b)
\]
for any $a,b\in[\alpha,\beta]$.  
We can then deduce \eqref{fffffff}  from \eqref{hh}. 
Thanks to \eqref{eqn:converge-uLambda}, we  obtain that,   
for any $\delta>0$ fixed,   
if $i$ is sufficiently large,  then 
\begin{equation}\label{zzz}
\int_{(K,\omega_i)}\left\langle \Psi(u_i)(\partial u_i), \partial u_i\right\rangle_i + \int_{(X,\omega_\infty)} \Tr \big(u_\infty\Lambda_{\infty}\Theta_h\big) \leq \delta.
\end{equation}

Since  $u_i\rightarrow u_\infty$ in $L^2_{b}$ on $(K,\omega_\infty)$,  
for some $b$ depending on $K$, 
we can apply the item $(2)$ of Lemma \ref{lemma:simpsonmor}  to show that,   
there is a convergence 
\[
\Psi^{\frac{1}{2}}(u_i)\rightarrow \Psi^{\frac{1}{2}}(u_\infty)    \mbox{ when } i\to \infty, 
\] 
in $\mathcal{C}^0(L^2, L^q)$ for any $q<2$, 
where $\Psi^{\frac{1}{2}}$ is the positive square root of $\Psi$, which is again smooth.   
Hence, from \eqref{zzz}, we  deduce  that, for all $i$ sufficiently large, 
\[
\| \Psi^{\frac{1}{2}}(u_\infty)(\partial u_i)\|^2_{L^q(K,\omega_i)}+ \int_{(X,\omega_\infty)} \Tr \big(u_\infty\Lambda_{\infty}\Theta_h\big) \leq 2\delta.\]
In addition, we have the following weak convergence in $L^q (K,\omega_\infty)$ 
\[
\Psi^{\frac{1}{2}}(u_\infty)(\partial u_i)\rightarrow \Psi^{\frac{1}{2}}(u_\infty)(\partial u_\infty) \mbox{ when }  i\to \infty. 
\]  
By the Hahn-Banach theorem,  the previous inequality implies that 
\[
\| \Psi^{\frac{1}{2}}(u_\infty)(\partial u_\infty)\|^2_{L^q(K,\omega_\infty)}+ \int_{(X,\omega_\infty)} \Tr \big(u_\infty\Lambda_{\infty}\Theta_h\big) \leq 2\delta.
\]

This inequality holds for any $\delta>0$ and any $q<2$. 
If a measurable function satisfies an 
$L^q$ norm inequality which is uniform for $q<2$ then it satisfies the inequality 
for $q=2$.  
Note that $K$ can be arbitrarily large in $X\setminus \pi (D)$. 
Hence we obtain the item $(2)$ of the claim.
This  completes the proof of the claim. 
\end{proof}

\subsection{Equality condition for Bogomolov-Gieseker inequalities}  

We complete the proof of Theorem \ref{thm:BG-inequality-intro} in this subsection.  

\begin{thm}\label{mainthm:con} 
For the sequence of Hermitian-Einstein metrics  $h_{\epsilon,HE}$, we have
\begin{enumerate}
\item 
$\int_{(X,\omega_\epsilon)}\|\Theta_{h_{\epsilon,HE}}\|_{\omega_\epsilon}^2  \leq C $ for some constant $C$ is independent of $\epsilon$, where $\Theta$ represents the Chern curvature tensor.    
\item There is a Zariski open set $U\subset X_{\sm} \setminus \pi(D)$ whose complement has codimension at least 2, 
there is a sequence $\{\epsilon_i\}$ of positive small enough numbers converging to $0$, 
such that $h_{\epsilon_i,HE}$ converge to a Hermitian-Einstein metric $h_{\infty}$ with respect to $\rho^*\omega_Z$, 
locally and smoothly on $U$.  
Moreover, $\Theta_{h_\infty}$ belongs to $L^2(X,\rho^*\omega_Z)$.  
\item Assume that $\hat{c}_2(\mathcal{F})\cdot [\omega_Z]^{n-2} = \hat{c}_1(\mathcal{F})^2  \cdot [\omega_Z]^{n-2} = 0$, 
then the Hermitian-Einstein metric $H_{\infty}$ defined on $U$   is Hermitian flat. 
\item  Assume the condition of   (3) holds, and that $Z$ has klt singularities.  
Then there is a finite quasi-\'etale cover $p\colon Z'\to Z$, such that the reflexive pullback $(p^*\mathcal{F})^{**}$ is a unitary flat vector bundle. 
\end{enumerate}
\end{thm}

\begin{proof}  

We recall the following  identity (see for example the proof of  \cite[Theorem 4.4.7]{Kobayashi2014}), where $c_n$ is a constant depending only on $n$, 
$$\Big(2\hat{c}_2(\mathcal{E}) -\hat{c}_1(\mathcal{E})^2 \Big) \cdot [\omega_\epsilon]^{n-2}=c_n\int (\|\Theta_{h_{\epsilon,HE}}\|_{h_{\epsilon,HE},\omega_\epsilon}^2-\|\Lambda_\epsilon \Theta_{h_{\epsilon,HE}}\|_{h_{\epsilon,HE}}^2)\omega_\epsilon^n.$$
The LHS is bounded by constants independent of $\epsilon$.  
The functions  $\|\Lambda_\epsilon \Theta_{h_{\epsilon,HE}}\|_{h_{\epsilon,HE}}^2$ are constant after the Einstein condition, and they are uniformly bounded as well.  
Hence $\|\Theta_{h_{\epsilon,HE}}\|_{L^2(X,\omega_\epsilon)}$ is uniformly bounded. 
This proves the item $(1)$.

We choose $U\subset X_{\sm} \setminus \pi(D)$ as the maximal Zariski open set over which $\rho^*\mathcal{F}$ is locally free.   
We note that $\rho|_U$ is isomorphic, and $\rho^*\omega_Z$ is a smooth K\"ahler form on $U$.  
Since we have   uniform $\mathcal{C}^0$ estimates for $H_\epsilon$ on any compact subsets of $U$ by Proposition \ref{prop:C0}, 
the convergence in the item $(2)$ is a standard consequence of the elliptic theory on the Hermitian-Einstein equations \eqref{eqn:HE}.   
The $L^2$ property for $\Theta_{h_\infty}$ follows from the item $(1)$ 


Now we prove the item $(3)$. By the Hermitian-Einstein condition, we have
\[ \Big(c_2(\cE,h_{\epsilon,HE})-\frac{r-1}{2r}c_1(\cE,h_{\epsilon,HE})^2 \Big)  \wedge  \omega_\epsilon^{n-2} \geq 0. \] 
On the other hand,  by assumption, we have 
$$\Big(2r\hat{c}_2(\mathcal{E}_{\orb}) - (r-1)\hat{c}_1(\mathcal{E}_{\orb})\Big)\cdot [\omega_\epsilon]^{n-2}\rightarrow 0  \mbox{ when } \epsilon\to 0.$$
By  the positivity of the integrands,   
for any precompact open subset $K\subset U$, we have
\[   \int_K \Big(c_2(\cE,h_{\epsilon,HE})-\frac{r-1}{2r}c_1(\cE,h_{\epsilon,HE})^2 \Big)  \wedge  \omega_\epsilon^{n-2} \rightarrow  0  \mbox{ when } \epsilon\to 0. \] 
Hence \[   \int_K \Big(c_2(\cF,h_\infty)-\frac{r-1}{2r}c_1(\cF,h_\infty)^2 \Big)  \wedge  \omega_Z^{n-2} = 0. \] 
This implies that the non negative integrand in the LHS is identically $0$. 
Since $K$ can be arbitrarily large in $U$, we deduce that  $h_{\infty}$ is a Hermitian flat. 

Finally, the item (4) follows from Theorem \ref{thm:klt-cover},  by using the argument of the proof of \cite[Theorem 1.14]{GrebKebekusPeternell2016b}. 
This completes the proof of the theorem. 
 \end{proof}

\begin{proof}[{Proof of Theorem \ref{thm:BG-inequality-intro}}] 
By Theorem \ref{mainthm:con} above, we can deduce that the item $(1)$ implies $(2)$. 
For the converse, we first note that $Z'$ also has klt singularities.  
We denote $\mathcal{F}'=(p^*\mathcal{F})^{**}$. 
Let $\rho'\colon X'\to Z'$ be an orbifold modification as in \cite[Theorem 1.2]{Ou2024}. 
Then $\mathcal{E}':=\rho'^*\mathcal{F}'$ is a unitary flat vector bundle on $X'$. 
If $\widetilde{X}$ is the universal cover of $X'$, then there is a trivial bundle $\widetilde{\mathcal{E}}$ with trivial Hermitian metric $\widetilde{h}$, such that $\mathcal{E}'\cong \widetilde{\mathcal{E}}/\pi_1(X')$ for some appropriate unitary representation of $\pi_1(X')$.  
It follows that the trivial metric $\widetilde{h}$ descend to some smooth flat metric $h'$ on $\mathcal{E'}$. Then $h'$ is orbifold smooth with respect to the standard orbifold structure on $X$. 
This implies that $\hat{c}_1(\mathcal{E}')=0$ and $\hat{c}_2(\mathcal{E}')=0$.   
It is then routine to verify the item $(1)$ of the theorem holds.  
\end{proof}



\bibliographystyle{alpha}
\bibliography{reference}

\end{document}